\documentclass[12pt,reqno]{amsart}
\usepackage[numbers,sort&compress]{natbib}
\usepackage{amsfonts,bbm}
\usepackage{amssymb,color}
\usepackage{fancyhdr}
\usepackage[titletoc]{appendix}
\usepackage{enumitem}
\usepackage{amsgen}
\usepackage{amscd}
\usepackage{amsmath}
\usepackage{mathrsfs}
\usepackage{cases}
\usepackage{accents} 

\usepackage{verbatim}
\usepackage{makecell}   

\usepackage{float}

\usepackage[colorlinks=true]{hyperref}
\hypersetup{urlcolor=blue, citecolor=green, linkcolor=blue}

\usepackage[margin=3cm, a4paper]{geometry}

\newtheorem{thm}{Theorem}[section]

\newtheorem{lem}[thm]{Lemma}
\newtheorem{prop}[thm]{Proposition}
\newtheorem{defn}[thm]{ \bf{Definition}}

\newtheorem{assu}[thm]{Assumption}
\theoremstyle{remark}
\newtheorem{remark}[thm]{Remark}

\newcommand{\EQ}[1]{\begin{align*}\begin{split} #1 \end{split}\end{align*}}
\newcommand{\EQn}[1]{\begin{align}\begin{split} #1 \end{split}\end{align}}

\newcommand{\Del}[1]{}

\def\norm#1{\left\|#1\right\|}
\def\normo#1{\|#1\|}
\def\normb#1{\big\|#1\big\|}

\def\normB#1{\Big\|#1\Big\|}

\def\abs#1{\left|#1\right|}

\def\absb#1{\big|#1\big|}

\def\brko#1{(#1)}
\def\brkb#1{\big(#1\big)}

\def\fbrk#1{\left\lbrace#1\right\rbrace}

\def\fbrkb#1{\big\lbrace#1\big\rbrace}
\def\fbrkbb#1{\bigg\lbrace#1\bigg\rbrace}

\def\jb#1{\langle#1\rangle}

\def\wt#1{\widetilde{#1}}
\def\wh#1{\widehat{#1}}
\def\wb#1{\overline{#1}}
\def\wtt#1{\accentset{\approx}{#1}} 

\def\pd{\partial}

\newcommand{\ra}{{\rightarrow}}

\def\loe{\le}
\def\goe{\ge}
\def\lsm{\lesssim}

\newcommand{\R}{{\mathbb R}}
\newcommand{\C}{{\mathbb C}}
\newcommand{\Z}{{\mathbb Z}}

\newcommand{\F}{{\mathcal{F}}}

\newcommand{\M}{{\mathcal{M}}}

\newcommand{\E}{{\mathcal{E}}}

\newcommand{\U}{{\mathcal{U}}}
\newcommand{\V}{{\mathcal{V}}}
\newcommand{\W}{{\mathcal{W}}}
\newcommand{\dd}{{\mathrm{d}}}

\newcommand{\TT}{{\mathcal{T}}}

\def\dx{\mathrm{\ d} x}
\def\dy{\mathrm{\ d} y}
\def\ds{\mathrm{\ d} s}

\newcommand{\re}{{\mathrm{Re}}}
\newcommand{\im}{{\mathrm{Im}}}

\def\al{\alpha}

\def\ph{\varphi}
\def\th{\theta}
\def\si{\sigma}
\def\de{\delta}
\def\De{\Delta}

\def\ta{\tau}
\def\la{\lambda}

\newcommand{\I}{\infty}
\def\rev#1{\frac{1}{#1}}
\def\half#1{\frac{#1}{2}}

\numberwithin{equation}{section}

\begin{document}
	\title[Final data problem for NLS in weighted spaces]{Global theory for NLS in the weighted spaces II: final data problem}

	
	\author{Yujin Guo}
	\address{(Y. Guo) Center for Applied Mathematics\\
		Tianjin University\\
		Tianjin 300072, China}
	\email{guoyujin1021@163.com}
	
	\author{Jia Shen}
	\address{(J. Shen) School of Mathematical Sciences and LPMC\\
		Nankai University\\
		Tianjin 300071, China}
	\email{shenjia@nankai.edu.cn}
	
	\author{Yifei Wu}
	\address{(Y. Wu) School of Mathematical Sciences\\
		Nanjing Normal  University\\
		Nanjing 210046, China}
	\email{yerfmath@gmail.com}
	
	\author{Changping Yang}
	\address{(C. Yang) Center for Applied Mathematics\\
		Tianjin University\\
		Tianjin 300072, China}
	\email{cp\_yang@tju.edu.cn}
	
	\date{}

	\begin{abstract}\noindent
		In this paper, we study the final data problem for the defocusing nonlinear Schr\"odinger equation (NLS) 
		\begin{equation*}
			i\pd_t u + \frac12\De u = |u|^p u
		\end{equation*}
		in weighted spaces \(\dot\Sigma^s(\mathbb R^d):=L_x^2(\mathbb R^d;|x|^{2s}\dd x)\). Let \(s_c=\frac d2-\frac2p\). Under the scaling of the NLS, \(\dot{\Sigma}^{-s_c}(\mathbb R^d)\) is the scaling-critical weighted space. In three dimensions, within the range of exponents considered in this paper, we show that the final data problem is locally well-posed for \(s\geq -s_c\) and ill-posed for \(s<-s_c\). This is opposite to the corresponding initial data problem, which is locally well-posed for \(s\leq -s_c\) and ill-posed for \(s>-s_c\).
		
		
		For the three-dimensional quadratic equation, we further obtain a unique global solution for every radial final datum in the critical space $\dot\Sigma^{\frac12}(\mathbb R^3)$. 
	\end{abstract}
	
	\maketitle
	
	
	
	\section{Introduction}
	
	\vspace{0.5cm}
	
	We consider the defocusing nonlinear Schr\"odinger equation (NLS):
	\begin{equation}
		\label{eq:nls}
		i\pd_t u + \frac12\De u = |u|^p u, 
	\end{equation}
	with $p>0$ and $u(t,x):I\times\R^d\ra\C$, where $I\subset\R$ is a time interval. The \emph{initial data problem} for \eqref{eq:nls} is the following: given an initial datum $u_0$, find a solution $u$ of \eqref{eq:nls} satisfying
	\begin{equation*}
		u(0,x)=u_0(x).
	\end{equation*} 
	The Cauchy theory in Sobolev spaces $H^s(\R^d)$ or $\dot H^s(\R^d)$ has been developed extensively; see, for example, \cite{GV79Cauchy,Kat87poincare,CW90critical,Caz03book} and the references therein. If a solution is global, a fundamental question is whether its large-time behavior is asymptotically linear. This leads to the notion of \textit{scattering}. More precisely, let $X$ be a Banach space. We say that a solution $u$ to \eqref{eq:nls} \textit{scatters} to $u_\pm\in X$ as
	$t\to\pm\infty$ if
	\begin{equation}\label{eq:final-data}
		\lim_{t\to\pm\infty}
		\bigl\|e^{-\frac12it\Delta}u(t)-u_\pm\bigr\|_X=0.
	\end{equation}
	The scattering theory for nonlinear Schr\"odinger equations has a long history \cite{GV79Scattering,Str81JFA,Caz03book,KTV14book}.
	
	In this paper, we study the inverse problem suggested by \eqref{eq:final-data}. Given a final datum $u_+$, one seeks a solution, defined for all sufficiently large times, that satisfies \eqref{eq:final-data} as $t\to+\infty$. We refer to this as the \emph{final data problem}. In contrast to the initial data problem, the datum is prescribed at temporal infinity rather than at a finite time. We first formulate the corresponding local theory near $t=+\infty$.
	
	\begin{defn}
		We say that the final data problem for \eqref{eq:nls} is \textit{locally well-posed} in $X$ if, for every \(u_+\in X\), there exists \(T>0\) and a unique solution \(u\) to \eqref{eq:nls} satisfying $e^{-\frac12it\Delta}u\in C([T,+\infty);X)$
		and scattering to \(u_+\) in \(X\) as \(t\to+\infty\). Moreover, the corresponding solution map $u_+
		\longmapsto
		e^{-\frac12it\Delta}u$, locally defined from $X$ to $C([T,+\infty);X)$, is uniformly continuous. The case \(t\to-\infty\) is analogous.
	\end{defn}
	
	The indication of the natural regularity for either the initial data problem or the final data problem comes from scaling. Equation \eqref{eq:nls} is
	invariant under
	\begin{equation*}
		u(t,x)\longmapsto
		u_\lambda(t,x):=\lambda^{\frac2p}u(\lambda^2t,\lambda x),
		\quad \lambda>0.
	\end{equation*}
	Set 
	\begin{equation*}
		s_c:=\frac d2-\frac2p.
	\end{equation*}
	Then, for every $s\in\mathbb R$,
	\begin{equation*}
		\|u_\lambda(0)\|_{\dot H_x^s(\mathbb R^d)}
		=\lambda^{s-s_c}
		\|u(0)\|_{\dot H_x^s(\mathbb R^d)}.
	\end{equation*}
	Thus, $\dot H^{s_c}(\mathbb R^d)$ is the scaling-critical Sobolev space. In the standard \(H^s\)-based theory for the initial data problem, the scaling index $s_c$ separates the expected regimes. Under the usual regularity assumptions on the nonlinearity, local well-posedness holds in the scaling-subcritical range \(s>s_c\); see \cite{CW90critical,Kat87poincare}. By contrast, in many scaling-supercritical cases \(s<s_c\), the flow map fails to be uniformly continuous, and stronger forms of ill-posedness may occur; see \cite{CCT03illposed}. The critical case \(s=s_c\) is generally more delicate \cite{CW89remarks,CKSTT08Annals,Vis07Duke}.

	For the final data problem considered here, since the asymptotic profile is prescribed at temporal infinity, its spatial localization plays a central role. This naturally leads us to consider the weighted spaces
	\begin{equation*}
		\Sigma^s(\mathbb R^d):=L_x^2(\mathbb R^d;\langle x\rangle^{2s}\,dx),
		\quad
		\dot\Sigma^s(\mathbb R^d):=L_x^2(\mathbb R^d;|x|^{2s}\,dx),
	\end{equation*}
	where \(s>0\) and \(\langle x\rangle=(1+|x|^2)^{\frac{1}{2}}\). If \(u\) scatters to \(u_+\) as \(t\to+\infty\), then the rescaled solution \(u_\lambda\) scatters to the rescaled final datum
	\begin{equation*}
		u_{+,\lambda}(x):=\lambda^{\frac{2}{p}}u_+(\lambda x).
	\end{equation*}
	Moreover,
	\begin{equation*}
		\|u_{+,\lambda}\|_{\dot\Sigma^s(\mathbb R^d)}
		=\lambda^{-s_c-s}\|u_+\|_{\dot\Sigma^s(\mathbb R^d)}.
	\end{equation*}
	Hence, \(\dot\Sigma^{-s_c}(\mathbb R^d)\) is the scaling-critical weighted space for the final data problem. We are particularly interested in the mass-subcritical regime \(s_c<0\), since the negative Sobolev critical index \(s_c\) corresponds to the positive weighted exponent \(-s_c\).

	Weighted spaces have played a classical role in the large-time analysis and scattering theory for NLS; see, for example, \cite{CW89remarks,CW92CMP,HO88poinc,Tsu85poincare} and the references therein. For mass-subcritical NLS above the Strauss exponent, Lee \cite{Lee19IMRN} showed that, for some \(0<\beta<p\), the data-to-scattering-state map does not admit a \(C^{1+\beta}\) extension from \(L^2\) to \(L^2\) near the origin. More recently, Burq, Georgiev, Tzvetkov, and Visciglia \cite{BGTV21NLS} established scattering for mass-subcritical NLS with initial data in \(\Sigma^1\cap H^1(\mathbb R^d)\). Shen and Wu \cite{SW23subNLS} further obtained scattering for large initial data in lower-weight spaces \(\Sigma^s\) for certain \(s<1\), as well as almost sure scattering for randomized \(L^2\) data. In the short-range mass-subcritical regime \(s_c<0\) and \(p>\frac{2}{d}\), Nakanishi \cite{Nak01siam} constructed global solutions to the final data problem in \(L^2_x(\mathbb R^d)\) and observed that the problem is supercritical in the \(L^2\) topology.
	
	These results, however, do not determine how the local final-data theory in the homogeneous weighted space \(\dot{\Sigma}^s(\R^d)\) depends on \(s\), nor whether the scaling-critical index \(s=-s_c\) separates the well-posed and ill-posed regimes. This is the first main question addressed in the present paper.
	
	In our previous work \cite{GSY26NLS-I}, we studied the initial data problem for \eqref{eq:nls} in $\dot\Sigma^s(\R^d)$. Schematically, the result is
	\begin{itemize}
		\item the initial data problem is locally well-posed when $s\le -s_c$;
		\item the initial data problem is ill-posed when $s>-s_c$.
	\end{itemize}
	However, the final data problem displays the opposite behavior. Within the parameter ranges covered below,
	\begin{itemize}
		\item the final data problem is locally well-posed when $s\ge -s_c$;
		\item the final data problem is ill-posed when $s<-s_c$.
	\end{itemize}
	Thus, $s=-s_c$ is the critical weighted threshold, but the subcritical and supercritical regimes are interchanged when the Cauchy data are moved from the initial time to time infinity. 

	The precise three-dimensional statement is as follows.

	\begin{thm}[Local theory of the final data problem]\label{thm:final-local}
		Let $0<s<1$ and $u_+ \in \dot\Sigma^s(\R^3)$. Consider \eqref{eq:nls} with final data condition
		\begin{equation}\label{eq:final-local}
			\lim_{t\ra+\I} \normb{e^{-\frac12it\De}u-u_+}_{\dot \Sigma^s(\R^3)} = 0.
		\end{equation}
		Then the following statements hold: 
		\begin{enumerate}
			\item 
			If \( s \geq -s_c \) and $ p\leq \frac{4}{3-2s}$, then there exists $T>0$ such that the equation \eqref{eq:nls} with \eqref{eq:final-local} admits a unique solution $u$ such that
			\begin{equation*}
				e^{-\frac{1}{2}it\Delta} u(t)\in C([T,+\I);\dot\Sigma^s(\R^3)).
			\end{equation*}
			Moreover, 
			the solution map $u_+
			\longmapsto
			e^{-\frac12it\Delta}u$, locally defined from $\dot\Sigma^s(\R^3)$ to $C([T,+\infty);\dot\Sigma^s(\R^3))$, is Lipschitz continuous when \(p\ge1\) and Hölder continuous of order \(p\) when \(0<p<1\), and hence uniformly continuous.
			\item 
			If \(s<-s_c\) and $\frac{2}{3}<p\leq 1$, then the equation \eqref{eq:nls} with condition \eqref{eq:final-local} is ill-posed in the following sense: for any $0<\varepsilon, \delta<1$ and for any $T>0$ sufficiently large there exist solutions $u^1, u^2$ of \eqref{eq:nls} with final data $u^1_+$, $u^2_+ \in \mathcal{S}(\mathbb{R}^3)$ such that
			\begin{align*}
				\|u^1_+\|_{\dot\Sigma^s(\R^3)}+\|u^2_+\|_{\dot\Sigma^s(\R^3)} &\leq C \varepsilon,\\
				\|u^1_+ -u^2_+\|_{\dot\Sigma^s(\R^3)}&< C \delta,\\
				\|e^{-\frac{1}{2}iT\Delta} u^1(T) - e^{-\frac{1}{2}iT\Delta} u^2(T)\|_{\dot\Sigma^s(\R^3)}&>c\varepsilon,
			\end{align*}
			where $C, c>0$ are constants independent of $\varepsilon$ and $\delta$. Thus, the solution map fails to be uniformly continuous on any neighborhood of the origin in $\dot\Sigma^s(\R^3)$.
		\end{enumerate}
	\end{thm}
	
	\begin{remark}
		(i) Although we state the result only in three dimensions, the argument can be surely extended to more general dimensions. 
		(ii) The instability in part (2) is not of the usual norm-inflation type (see \cite{CCT03illposed}), although it may be regarded as a weak form of norm inflation. 
	\end{remark}

	
	Theorem \ref{thm:final-local}, together with our previous result \cite{GSY26NLS-I}, yields the comparison summarized in Table~\ref{tab:threshold}. In particular, both the initial and final data problems are locally well-posed at the critical index $s=-s_c$.
	
	\begin{table}[H]
		\centering
		\begin{tabular}{|c|c|c|c|}
			\hline
			& $s<-s_c$ & $s=-s_c$ & $s>-s_c$ \\
			\hline
			initial data problem &
			\makecell{well-posed} &
			\makecell{well-posed} &
			\makecell{ill-posed} \\
			\hline
			final data problem &
			\makecell{ill-posed} &
			\makecell{well-posed} &
			\makecell{well-posed}  \\
			\hline
		\end{tabular}
		\caption{Schematic local theory for NLS in $\dot\Sigma^s(\R^3)$}\label{tab:threshold}
	\end{table}
	
	Theorem \ref{thm:final-local} constructs solutions only near temporal infinity. A natural next question is whether these solutions can be continued backward through every finite positive time, thereby yielding global solutions to the final data problem on \((0,+\infty)\). Such an extension is a necessary step toward constructing the forward \textit{wave operator}
	\begin{equation*}
		\Omega_+:u_+\longmapsto u(0).
	\end{equation*}
	It is not, however, sufficient by itself: one must additionally show that the solution extends continuously to \(t=0\) in the underlying space \(X\). For classical results on wave operators in weighted spaces, see, for example, \cite{CW92CMP,GOV94poincare,HO88poinc} and the references therein.
	
	In the present weighted setting, global continuation is nontrivial because no known conservation law directly controls the fractional weighted norm \(\dot{\Sigma}^s(\mathbb R^d)\). This is the main obstruction to extending the local solution given by Theorem \ref{thm:final-local}.

	For final data in $\Sigma^1(\R^3)$, the pseudo-conformal transform converts the problem into an $H^1$ initial data problem for a nonautonomous NLS. In the range considered below, a monotone pseudo-conformal energy supplies the missing a priori bound and allows the local solution to be continued globally. This yields our second main result, which requires no Sobolev regularity of the final data.


	\begin{prop}[Finite pseudo-conformal energy implies global solution]\label{thm:gwp}
		Let $\frac{4}{3}<p\leq 4$, equivalently \(0<s_c\leq 1\). Then, for any $u_+\in\Sigma^1(\R^3)$, there exists a unique global solution $u$ to \eqref{eq:nls} with $e^{-\frac{1}{2}it\Delta} u\in C((0,+\I);\Sigma^1(\R^3))$ such that
		\begin{equation*}
			\lim_{t\ra+\I} \normb{e^{-\frac{1}{2}it\Delta} u-u_+}_{\Sigma^1(\R^3)} = 0.
		\end{equation*}
	\end{prop}	
	
	The preceding argument applies to the noncritical weighted setting. At the scaling-critical regularity, however, this argument no longer provides the control required to extend the local final-data solution globally. We nevertheless overcome this difficulty for the radial three-dimensional quadratic NLS. Indeed, when \(d=3\) and \(p=1\), we have \(s_c=-\frac12\), so that \(\dot{\Sigma}^{\frac12}(\mathbb R^3)\) is scaling critical. Global well-posedness for the corresponding initial data problem with radial data was established in \cite{SW24quaNLS}. Our next main result establishes the corresponding global final-data theory at the same critical regularity.
	
	\begin{thm}[3D quadratic NLS]\label{thm:final}
		Let \(d=3\) and $p=1$. Suppose that $u_+\in \dot\Sigma^{\frac{1}{2}}(\R^3)$ and $u_+$ is radial. Then, equation \eqref{eq:nls} admits a unique global solution $u$ such that $e^{-\frac12it\De}u\in C((0,+\I); \dot\Sigma^{\frac{1}{2}}(\R^3))$ and 
		\begin{equation*}
			\lim_{t\ra+\I} \normb{e^{-\frac12it\De}u-u_+}_{\dot\Sigma^{\frac{1}{2}}(\R^3)} = 0.
		\end{equation*}
		Moreover, the solution has additional weighted regularity: for any $0\le s\le \frac12$,
		\begin{equation*}
			e^{-\frac12it\De}u - u_+ \in C((0,+\I); \dot\Sigma^{s}(\R^3)).
		\end{equation*}
	\end{thm}
	
	\begin{remark}
		Theorem \ref{thm:final} can be viewed as the backward-time counterpart of the corresponding result in \cite{SW24quaNLS}. Although the two proofs are parallel, the proof of Theorem \ref{thm:final} relies on an almost-conservation law for the mass rather than on the pseudo-conformal energy used in Proposition \ref{thm:gwp}.
	\end{remark}

	\begin{remark}[Asymptotic completeness]
		It is natural to conjecture that asymptotic completeness holds for the three-dimensional defocusing quadratic NLS in the critical space $\dot\Sigma^{\frac12}(\mathbb R^3)$. More precisely, one expects the wave operators
		\begin{equation*}
			\Omega_\pm:u_\pm\longmapsto u(0)
		\end{equation*}
		to be well-defined and onto $\dot\Sigma^{\frac12}(\mathbb R^3)$. Killip, Masaki, Murphy, and Visan \cite{KMMV17NoDEA} proved that a solution is global and scatters whenever its critical weighted $\dot\Sigma^{\frac12}$-norm remains bounded on its maximal lifespan. Consequently, asymptotic completeness would follow once such a bound is established for every maximal-lifespan solution.
	\end{remark}

	\subsection{Outline}
	The remainder of the paper is organized as follows. Section \ref{sec:prel} collects the notation and preliminary estimates. In Section \ref{sec:local result}, we prove the local well-posedness and ill-posedness statements in Theorem \ref{thm:final-local}. Section \ref{sec:finite pce} proves Proposition~\ref{thm:gwp} by means of the pseudo-conformal energy, while Section \ref{sec: 3DquaNLS} proves the critical global final-data result, Theorem~\ref{thm:final}.

	\vspace{2cm}
	
	\section{Preliminaries}\label{sec:prel}
	
	\vspace{0.5cm}
	
	\subsection{Basic notation} 
	For any $z\in\C$, we define $\re z$ and $\im z$ as the real and imaginary parts of $z$, respectively. 
	$C>0$ represents some constant that may vary from line to line. We write $C(a)>0$ for some constant depending on parameter $a$. If $f\loe C g$, we write $f\lsm g$. If $f\loe C g$ and $g\loe C f$, we write $f\sim g$. If $C=C(a)$ depends on $a$, then we write $f\lsm_a g$ and $f\sim_a g$, respectively. 
	
	
	We use $\wh f$ or $\F f$ to denote the Fourier transform of $f$:
	\begin{equation*}
		\wh f(\xi)=\F f(\xi):= \rev{(2\pi)^{d/2}}\int_{\R^d} e^{-ix\cdot\xi}f(x)\rm dx.
	\end{equation*}
	We also define the inverse Fourier transform:
	\begin{equation*}
		\F^{-1} g(x):= \rev{(2\pi)^{d/2}}\int_{\R^d} e^{ix\cdot\xi}g(\xi)\rm d\xi.
	\end{equation*}
	Using the Fourier transform, we can define the fractional derivative $\abs{\nabla} := \F^{-1}|\xi|\F $ and $\abs{\nabla}^s:=\F^{-1}|\xi|^s\F $.  

	We denote by $e^{\frac12it\De}u_0$ the solution to 
	\EQ{
		\left\{ \aligned
		&i\pd_t u + \frac12\De u = 0, \\
		& u(0,x) = u_0,
		\endaligned
		\right.
	}
	and set $S(t):=e^{\frac12it\De}$. Then, we have the explicit formula
	\begin{equation}\label{linear-explicit}
		S(t)u_0 = \frac{1}{(2\pi it)^{d/2}} \int_{\R^d} e^{i\frac{|x-y|^2}{2t}} u_0(y)\dy .
	\end{equation}
	
	Given $1\loe p \loe \I$, $L^p(\R^d)$ denotes the usual Lebesgue space. For $1\le p<\I$ and $1\le q\le\I$, the Lorentz space $L^{p,q}(\R^d)$ is defined by its quasi-norm
	\begin{equation*}
		\norm{f}_{L^{p,q}(\R^d)} := p^{\frac1q} \normb{ \la \absb{\fbrk{x\in\R^d:|f(x)|>\la}}^{\frac 1p} }_{L^q((0,\I),\frac{\dd\la}{\la})}.
	\end{equation*}
	We use the convention \(L^{\infty,q}:=L^\infty\).

	We call an exponent pair $(q,r)\in\R^2$ \textit{admissible} if $\frac{2}{q}+\frac{d}{r}=\half d$, $2\loe q\loe\I$, $2\loe r\loe\I$, and $(q,r,d)\ne(2,\I,2)$. We also say $(q,r)$ is \textit{acceptable} if $1\leq q< \infty$ and $\frac{1}{q}<\frac{d}{2}-\frac{d}{r}$, or $(q,r)=(\infty, 2)$.

	\subsection{Pseudo-conformal transform}
	Now, we recall some notation and the properties of the pseudo-conformal transform. We define the pseudo-conformal  transform $\mathcal T$ to be 
	\begin{equation}\label{def:PC-transform}
		\mathcal T f(t,x) := \frac{1}{(it)^{d/2}} e^{\frac{i|x|^2}{2t}} \bar f\left( \frac{1}{t},\frac{x}{t} \right).
	\end{equation}
	Then the inverse pseudo-conformal transform satisfies  
	$$
	\mathcal T^{-1}=\mathcal T.
	$$
	
	We define
	\begin{equation*}
		M(t) := e^{\frac{i|x|^2}{2t}}\text{, and } J(t):=x+it\nabla.
	\end{equation*}
	Recall the explicit formula for $S(t)$ in \eqref{linear-explicit}:
	\begin{equation*}
		S(t)u_0 = \frac{1}{(2\pi it)^{d/2}} \int_{\R^d} e^{i\frac{|x-y|^2}{2t}} u_0(y)\dy = \frac{1}{(it)^{d/2}} M(t) \F(M(t)u_0)\left( \frac xt \right).
	\end{equation*}
	This yields that for $f(x):\R^d\ra\C$,
	\begin{equation}\label{eq:pc-TS=SF}
		\TT S(t)f=S(t)\F^{-1}\wb f,
	\end{equation}
	and the following equivalent representations of the vector field:
	\begin{equation*}
		J(t)=S(t)xS(-t)=M(t)(it\nabla)M(-t).
	\end{equation*}
	Therefore, we define the fractional vector field for any $s>0$,
	\begin{equation*}
		|J(t)|^{s}:= S(t)|x|^sS(-t).
	\end{equation*}
	Similarly, $|J(t)|^{s} = M(t)|t\nabla|^sM(-t)$. Using this formula, for a spacetime function $u(t,x)$, we have
	\begin{equation}\label{eq:pc-nablaT=SJ}
		|\nabla|^s\TT u =\TT |J(t)|^s u.
	\end{equation}
	
	Now, we give two basic facts for the pseudo-conformal transform:
	\begin{remark}
		By \eqref{eq:pc-TS=SF}, if $u$ is the solution of \eqref{eq:nls} with final data $u_+\in\dot\Sigma^s(\R^d)$, then $\U=\TT u$ solves
		\begin{equation}\label{eq:nls-initial}
			i\pd_t \U + \frac12 \De \U = t^{\frac{dp}{2}-2}|\U|^p\U,
		\end{equation}
		with $\U(0)=\F^{-1} \wb u_+\in \dot{H}_x^s(\R^d)$. Moreover, by \eqref{eq:pc-nablaT=SJ}, for any $0<a<b<\I$, we have that
		\EQ{
			\U \in C([a,b];\dot H^s(\R^d)) \quad \Longleftrightarrow & \quad |J(t)|^su \in C([b^{-1},a^{-1}];L_x^2(\R^d)) \\
			\Longleftrightarrow & \quad S(-t)u \in C([b^{-1}, a^{-1}];\dot \Sigma^s(\R^d)).
		}
	\end{remark}
	
	\begin{remark}\label{rem:spacetime-exponent-transform}
		Let $u(t,x)$ be a function and
		\begin{equation*}
			\U(s,y)=\TT u(s,y)=  \frac{1}{(is)^{d/2}} e^{\frac{i|y|^2}{2s}} \wb u \left( \frac{1}{s},\frac{y}{s} \right).
		\end{equation*}
		Then, for any $1\le q,r\le\I$, we have
		\begin{equation*}
			\norm{u}_{L_t^q L_x^r(\R\times\R^d)} = \normb{|s|^{\frac d2-\frac2q-\frac dr}\U(s,y)}_{L_s^q L_y^r(\R\times\R^d)}.
		\end{equation*}
	\end{remark}
	
	\subsection{Useful lemmas}
	
	\begin{lem}[Kato-Ponce's inequality, \cite{Li19RMI}]\label{lem:kato-ponce-inequality}
		Let $0<s<1$, $1<p\leq\infty$, and $1< p_1, p_2, p_3, p_4 \leq \infty$ with $\frac{1}{p}=\frac{1}{p_1}+\frac{1}{p_2}$ and $\frac{1}{p}=\frac{1}{p_3}+\frac{1}{p_4}$, then 
		\begin{equation*}
			\big\| |\nabla|^s(fg) \big\|_{L^p}\lesssim \big\| |\nabla|^s f \big\|_{L^{p_1}}\|g\|_{L^{p_2}}+\big\| |\nabla|^s g \big\|_{L^{p_3}}\|f\|_{L^{p_4}}.
		\end{equation*}
	\end{lem}
	
	\begin{lem}[Fractional chain rule, \cite{Vis07Duke}]\label{lem:Frac_chain-alphaorder}
		Let $G$ be a H\"older continuous function of order $0<\alpha<1$. Then, for every $0<s<\alpha$, $1<p<\I$, and $\frac s \al<\sigma<1$, we have
		\begin{equation*}
			\norm{\abs{\nabla}^sG(u)}_{L_x^p}\lsm\norm{|u|^{\al-\frac s \si}}_{L_x^{p_1}}\norm{\abs{\nabla}^{\si}u}_{L_x^{\frac{s}{\si}p_2}}^{\frac{s}{\si}},
		\end{equation*}
		provided $\rev p=\rev{p_1}+\rev{p_2}$ and $(1-\frac{s}{\al\si})p_1>1$.
	\end{lem}
	
	\begin{lem}[Derivatives of differences, \cite{2011-Killip-Visan}]\label{2: Lemma-H^s,0<s<1}
		Set $F(z)=|z|^pz$, with $p >0$. Let $s\in (0,1)$ and $ 1<r_1, r_2, r_3 < \infty$ satisfy $\frac{1}{r_1} = \frac{1}{r_2} + \frac{1}{r_3}$. Then we have 
		\begin{equation*}
			\||\nabla|^s(F(u)-F(v))\|_{L^{r_1}} \lesssim \|u\|_{L^{pr_2}}^p \||\nabla|^s (u-v)\|_{L^{r_3}}
			+ \||\nabla|^sv\|_{L^{r_3}}\|u-v\|^p_{L^{pr_2}}.
		\end{equation*}
	\end{lem}
	
	\begin{lem}[Strichartz estimate, \cite{KT98AJM,KTV14book}]\label{lem:strichartz}
		Let $a\in I\subset \R$. Suppose that $(q,r)$ and $(\wt{q},\wt{r})$ are admissible.
		Then,
		\begin{equation}\label{eq:strichartz-1}
			\norm{ S(t)\ph}_{L_t^{\infty}L_x^2 \cap L_t^qL_x^r(I\times\R^d)} \lsm \norm{\ph}_{L_x^2(\R^d)},
		\end{equation}
		and
		\begin{equation}\label{eq:strichartz-2}
			\left\| \int_a^t S(t-s) F(s)\ds \right\|_{L_t^{\infty}L_x^2\cap L_t^qL_x^r(I\times \R^d)} \lsm \norm{F}_{L_t^{\wt{q}'} L_x^{\wt{r}'}(I\times\R^d)}.
		\end{equation}
		The above estimates also hold when $L_t^q$ and $L_t^{\wt{q}'}$ are replaced by Lorentz space $L_t^{q,2}$ and $L_t^{\wt{q}', 2}$ with $q, \wt{q} \neq \infty$.
	\end{lem}
	
	\begin{lem}[Inhomogeneous Strichartz estimates, \cite{Mas15CPAA,F05JHDE}]\label{lem-inhomogeneous Strichartz estimates}
		Let $d\geqslant 3$ and $I\subset \mathbb{R}$. Suppose that $(q,r)$ and $(\tilde{q}, \tilde{r})$ are acceptable. Assume further that
		\begin{equation*}
			\frac{2}{q} + \frac{d}{r} +\frac{2}{\wt{q}} + \frac{d}{\wt{r}} =d;\quad \frac{1}{q} + \frac{1}{\wt{q}}<1 ;\quad \frac{d-2}{d} \leq \frac{r}{\wt{r}} \leq \frac{d}{d-2}.
		\end{equation*}
		Then, for $a\in I$, we have
		\begin{equation}\label{esti:inho-Strichartz-Lp}
			\left\|\int_{a}^{t}S(t-s)f(s)\mathrm{d}s\right\|_{L_t^{q}L_x^r(I\times \mathbb{R}^d)}\lesssim \|f\|_{L_t^{\tilde{q}^{\prime}} L_x^{\tilde{r}^{\prime}}(I\times \mathbb{R}^d)},
		\end{equation}
		and 
		\begin{equation}\label{esti:inho-Strichartz-Lpq}
			\left\|\int_{a}^{t}S(t-\tau)F(\tau)\mathrm{d} \tau \right\|_{L_t^{q, 2}L_x^{r}(I\times \mathbb{R}^d)} \lesssim \|F \|_{L_t^{\tilde{q}^{\prime}, 2}L_x^{\tilde{r}^{\prime}}(I\times \mathbb{R}^d)}.
		\end{equation}
	\end{lem}
	
	\begin{lem}[Radial Strichartz estimates, \cite{GW14JAM}]\label{RSE:radial strichartz estimate}
		Let $d\geq 3$, $\gamma \in \mathbb{R}$, $I\subset\mathbb{R}$, and let $u_0(x)$ be a radial function. Suppose that $(q,r)$ satisfy $2\leq q,r \leq \infty$, $(q,r)\neq (2,\frac{4d-2}{2d-3})$, $\frac{2}{q}+\frac{2d-1}{r} \leq \frac{2d-1}{2}$ and
		\begin{equation*}
			\frac{2}{q} + \frac{d}{r} = \frac{d}{2} - \gamma.
		\end{equation*}
		Then, we have
		\begin{equation}\label{eq: radial estimate}
			\|S(t)u_0\|_{L_t^{\infty}\dot{H}_x^{\gamma}\cap L_t^{q}L_x^{r}(I\times \mathbb{R}^d)} \lesssim \|u_0\|_{\dot{H}^{\gamma}(\mathbb{R}^d)}.
		\end{equation} 
	\end{lem}

	\vspace{2cm}
	
	\section{Local results}\label{sec:local result}
	
	\vspace{0.5cm}

	In this section, we prove Theorem \ref{thm:final-local}. By the pseudo-conformal transform $\mathcal{U}=\mathcal{T}u$, the equation \eqref{eq:nls} with final data \eqref{eq:final-local} is equivalent to \eqref{eq:nls-initial} with initial data $\U(0,x)=\mathcal{U}_0(x)= \F^{-1} \wb u_+\in \dot{H}_x^s(\R^3)$.
	\subsection{Local well-posedness}\label{section-lwp}
	By the pseudo-conformal transform, Theorem \ref{thm:final-local} $(1)$ is a consequence of the following proposition.
	\begin{prop}\label{prop:local result}
		Let $\U_0\in \dot H_x^s(\R^3)$ and \(0<s<1\). 
		\begin{enumerate}
			\item 
			Let $\frac{4}{3+2s}<p\leq \frac{4}{3-2s}$. Then there exists $T_0=T_0(\norm{\U_0}_{\dot H_x^s(\R^3)})>0$ such that the equation \eqref{eq:nls-initial} admits a unique solution $\U\in C([0, T_0];\dot H_x^s(\R^3))$ with $\U(0,x)=\U_0(x)$.
			\item 
			Let $p=\frac{4}{3+2s}$. Then there exists $T_0=T_0(\U_0)>0$ such that the equation \eqref{eq:nls-initial} admits a unique solution $\U\in C([0, T_0];\dot H_x^s(\R^3))$ with $\U(0,x)=\U_0(x)$.
		\end{enumerate}
		Moreover, in both cases, the corresponding solution map $\U_0
		\longmapsto \U$, locally defined from $\dot H_x^s(\R^3)$ to $C([0,T_0];\dot H_x^s(\mathbb R^3))$, is Lipschitz continuous when \(p\ge1\) and Hölder continuous of order \(p\) when \(0<p<1\).
	\end{prop}
	
	The case \(p\geq 1\) can be handled by the standard contraction mapping argument. We therefore turn to the range $\frac{4}{3+2s}\leq p<1$, which can occur only when \(s>\frac12\). In this regime, the nonlinearity \(F(z)=|z|^p z\) fails to be \(C^2\), which brings extra difficulty. We introduce the following auxiliary spacetime exponents.
	
	\begin{defn}\label{defn:lwp-parameters}
		Let $\frac{4}{3+2s}\leq p<1$ and $I\subset \R$.
		\begin{enumerate}
			\item We denote $\theta=\frac{1}{2}sp+\frac{1}{2}(s-\frac{(3-2s)p}{4})$, and $(q_0, r_0)$, $(\wt{q}, \wt{r})$ are given by
			\begin{equation}\label{eq:parameters}
				\begin{aligned}
					\frac{1}{q_0}&=\frac{(3-2s)p-4(s-\theta)}{4(p+2)}, \quad \frac{1}{r_0}=\frac{3+\theta p}{3(p+2)}, \\
					\frac{1}{\wt{q}}&=\frac{(3+2s-4\theta)p+4(s-\theta)}{4(p+2)}, \quad \frac{1}{\wt{r}}=\frac{3+\theta p}{3(p+2)}.
				\end{aligned}
			\end{equation}
			\item We denote
			\begin{align}\label{defn:x1y1}
				\|f\|_{X_1(I)}:=\||\nabla|^{\theta}f\|_{L_t^{q_0}L_x^{r_0}(I\times\mathbb{R}^3)},\quad \|f\|_{Y_1(I)}:=\||\nabla|^{\theta}f\|_{L_t^{\wt{q}^{\prime}}L_x^{\wt{r}^{\prime}}(I\times\mathbb{R}^3)};\\
				\|f\|_{X_2(I)}:=\||\nabla|^{\theta}f\|_{L_t^{q_0, 2}L_x^{r_0} (I\times\mathbb{R}^3)},\quad \|f\|_{Y_2(I)}:=\||\nabla|^{\theta}f\|_{L_t^{\wt{q}^{\prime},2}L_x^{\wt{r}^{\prime}}(I\times\mathbb{R}^3)}.
			\end{align}
		\end{enumerate}
		
	\end{defn}
	\begin{remark}
		Next, we show that the pairs $(q_0, r_0)$ and $(\wt{q}, \wt{r})$ defined above are acceptable and satisfy $s-\frac{(3-2s)p}{4}<\theta<sp$, $2< q_0, \wt{q} < \infty$, and
		\begin{equation}\label{eq:lwp-parameters-scaling}
			\frac{2}{q_0}+\frac{3}{r_0}=\frac{3}{2}-(s-\theta),\quad \frac{2}{\wt{q}}+\frac{3}{\wt{r}}=\frac{3}{2}+(s-\theta).
		\end{equation}
		First, \eqref{eq:lwp-parameters-scaling} follows from a direct calculation. We observe that $s-\frac{(3-2s)p}{4}< sp$ is equivalent to $\frac{4s}{3+2s}<p$. Then, by the choice of $\th$, we have that $s-\frac{(3-2s)p}{4}<\theta<sp$. It is easy to see that $q_0, \wt{q}>0$. Recalling  \eqref{eq:lwp-parameters-scaling}, we obtain
		\begin{equation*}
			\frac{1}{q_0} + \frac{3}{r_0} = \frac{3}{2} - (s-\th) -\frac{1}{q_0} <\frac 32.
		\end{equation*}
		Thus, $(q_0, r_0)$ is acceptable.
		Furthermore, by $(3-2s)p<4$ and $\th<sp<s$, we have that $2<q_0<\I$. Similarly, $(\wt{q}, \wt{r})$ is acceptable. Note that $\frac{(3-2s)p}{4}\leq \frac{1}{2}$. Then, we have that $s-2\theta<s-\theta<\frac{1}{2}$, which implies $2< \wt{q}<\infty$. The choice of $\theta, \ (q_0, r_0),$ and $(\wt{q}, \wt{r})$ is not unique.
	\end{remark}
	
	We derive some estimates that will help us control the nonlinearity.
	\begin{lem}[Nonlinear estimates]\label{lem: Nonlinear estimate}
		Let $I:=[0,T]$ and $F(z):=|z|^pz$. Assume that $\theta$, $(q_0, r_0)$, $(\wt{q}, \wt{r})$, $X_1$, $Y_1$, $X_2$ and $Y_2$ are given in Definition \ref{defn:lwp-parameters}. Let $\wt{r}_0\in (2, 6)$ such that $\frac{2}{q_0}+\frac{3}{\wt{r}_0}=\frac{3}{2}$. Then for $\frac{4}{3+2s}< p<1$,
		\begin{equation}\label{Nonlinear estimate -II}
			\begin{aligned}
				&\left\|\int_{0}^{t}S(t-\tau)(\tau^{\frac{3p}{2}-2} (F(\U) - F(\V)) ) \mathrm{d} \tau \right\|_{X_1(I)} \\
				&\lesssim T^{\frac{(3+2s)p-4}{4}} \big( \||\nabla|^s \U\|^p_{L_t^{q_0}L_x^{\wt{r}_0}(I\times\R^3)} + \||\nabla|^s\V \|^p_{L_t^{q_0}L_x^{\wt{r}_0}(I\times\R^3)}\big)\|\U-\V\|_{X_1(I)}.
			\end{aligned}
		\end{equation}
		and for $p=\frac{4}{3+2s}$,
		\begin{equation}\label{Nonlinear estimate -III}
			\begin{aligned}
				&\left\|\int_{0}^{t}S(t-\tau)(\tau^{\frac{3p}{2}-2} (F(\U) - F(\V)) ) \mathrm{d} \tau \right\|_{X_2(I)} \\
				&\lesssim \big( \||\nabla|^s \U\|^p_{L_t^{q_0,2}L_x^{\wt{r}_0}(I\times\R^3)} + \||\nabla|^s\V \|^p_{L_t^{q_0,2}L_x^{\wt{r}_0}(I\times\R^3)}\big)\|\U-\V\|_{X_2(I)}.
			\end{aligned}
		\end{equation}
	\end{lem}
	\begin{proof}
		By \eqref{esti:inho-Strichartz-Lp},
		\begin{equation}\label{esti:X1-Y1}
			\left\|\int_{0}^{t}S(t-\tau)f(\tau, \cdot) \mathrm{d} \tau \right\|_{X_1(I)}\lesssim \|f\|_{Y_1(I)}.
		\end{equation}
		In the following, we restrict \((t,x) \) to $I\times\mathbb{R}^3$. We denote 
		\begin{equation*}
			\frac{1}{q_1}=\frac{4-(3-2s)p}{4}; \quad \frac{1}{r_1}=\frac{1}{r_0}-\frac{\theta}{3}
		\end{equation*}
		Note that
		\begin{equation*}
			F(z)-F(w)=(z-w)\int_{0}^{1}F_z(w+\rho(z-w))\mathrm{d}\rho+(\bar{z}-\bar{w})\int_{0}^{1}F_{\bar{z}}(w+\rho(z-w))\mathrm{d}\rho,
		\end{equation*}
		where $F_z=\frac{p+2}{2}|z|^p, F_{\bar{z}}=\frac{p}{2}|z|^{p-2}z^2$, which are both H\"older continuous of order $p$. Therefore, by \eqref{esti:X1-Y1}, the claim \eqref{Nonlinear estimate -II} will hold once we show that
		\begin{equation}\label{Lemma: Nonlinear estimate II 3}
			\begin{aligned}
				&\big\|t^{\frac{3p}{2}-2}|\nabla|^{\theta}(F_z(u+v)w)\big\|_{L_t^{\wt{q}^{\prime}}L_x^{\wt{r}^{\prime}}}\\
				&\lesssim \|t^{\frac{3p}{2}-2}\|_{L_t^{q_1}(I)} \big(\||\nabla|^s u\|^p_{L_t^{q_0}L_x^{\wt{r}_0}} + \||\nabla|^s v\|^p_{L_t^{q_0}L_x^{\wt{r}_0}}\big)\||\nabla|^{\theta}w\|_{L_t^{q_0}L_x^{r_0}}.
			\end{aligned}
		\end{equation}
		
		Now, it suffices to prove \eqref{Lemma: Nonlinear estimate II 3}. Take $q_2$, $r_2$ and $r_3$ such that
		\begin{equation*}
			\frac{1}{q_2}=\frac{p}{q_0}; \quad \frac{1}{r_2}=\frac{p}{r_1}+\frac{\theta}{3}; \quad \frac{1}{r_3}=(p-\frac{\theta}{s})\frac{1}{r_1}.
		\end{equation*}
		By Lemma \ref{lem:kato-ponce-inequality},
		\begin{align*}
			\big\|t^{\frac{3p}{2}-2}|\nabla|^{\theta}(F_z(u+v)w)\big\|_{L_t^{\tilde{q}^{\prime}}L_x^{\tilde{r}^{\prime}}}&\lesssim T^{\frac{(3+2s)p-4}{4}} \||\nabla|^{\theta}F_z(u+v)\|_{L_t^{q_2}L_x^{r_2}}\|w\|_{L_t^{q_0}L_x^{r_1}}\\
			& \qquad +T^{\frac{(3+2s)p-4}{4}} \|u+v\|^p_{L_t^{q_0}L_x^{r_1}}\||\nabla|^{\theta}w\|_{L_t^{q_0}L_x^{r_0}}.
		\end{align*}
		Thus, by Lemma \ref{lem:Frac_chain-alphaorder}, H\"older's and Sobolev's inequalities,
		\begin{align*}
			\||\nabla|^{\theta}F_z(u+v)\|_{L_t^{q_2}L_x^{r_2}}
			&\lesssim \big\|\||u+v|^{p-\frac{\theta}{s}} \|_{L_x^{r_3}} \||\nabla|^{s}(u+v) \|^{\frac{\theta}{s}}_{L_x^{\wt{r}_0}} \big\|_{L_t^{q_2}}\\
			&\lesssim \|u+v\|^{p-\frac{\theta}{s}}_{L_t^{q_0}L_x^{r_1}} \| |\nabla|^{s}(u+v) \|^{\frac{\theta}{s}}_{L_t^{q_0}L_x^{\wt{r}_0}} \\
			&\lesssim \||\nabla|^s u\|^p_{L_t^{q_0}L_x^{\tilde{r}_0}} + \||\nabla|^s v\|^p_{L_t^{q_0}L_x^{\tilde{r}_0}}.
		\end{align*}
		This finishes the proof of \eqref{Lemma: Nonlinear estimate II 3} and hence \eqref{Nonlinear estimate -II}.
		
		The same argument, with \eqref{esti:inho-Strichartz-Lpq} in place of \eqref{esti:inho-Strichartz-Lp}, proves \eqref{Nonlinear estimate -III}.
	\end{proof}
	
	We are now in a position to prove Proposition \ref{prop:local result}.
	\begin{proof}[Proof of Proposition \ref{prop:local result}]
		Let $I:=[0, T_0]$. Define the map
		\begin{equation*}
			\Phi(\U):=S(t)\U_0-i\int_{0}^{t}S(t-\tau)(\tau^{\frac{3p}{2}-2}|\U|^p\U)\mathrm{d}\tau.
		\end{equation*}
		If $p\geq 1$, we set $\frac{1}{q_0}= \frac{(3-2s)p}{4(p+2)}$. If $p<1$, $\frac{1}{q_0}$ is defined in \eqref{eq:parameters}. Let $\wt{r}_0\in (2,6)$ such that $\frac{2}{q_0}+\frac{3}{\wt{r}_0}=\frac{3}{2}$. 
		
		We first consider $\frac{4}{3+2s}<p\leq \frac{4}{3-2s}$. We denote the resolution space
		\begin{equation*}
			E_1(R):=\{ \U\in L_t^{q_0}\dot{W}_x^{s,\wt{r}_0}(I\times \mathbb{R}^3) :\||\nabla|^s \U\|_{L_t^{q_0}L_x^{\wt{r}_0}(I\times \mathbb{R}^3)}\leqslant 2R \}.
		\end{equation*}
		When $1\leq p \leq \frac{4}{3-2s}$, we consider the metric
		\begin{equation*}
			d_1(\U, \V)=\||\nabla|^s (\U-\V)\|_{L_t^{q_0}L_x^{\wt{r}_0}(I\times \mathbb{R}^3)}.
		\end{equation*}
		When $\frac{4}{3+2s}< p <1$, we consider the metric
		\begin{equation*}
			d_2(\U,\V)=\|\U-\V\|_{X_1(I)}=\||\nabla|^{\theta} (\U-\V)\|_{L_t^{q_0}L_x^{r_0}(I\times \mathbb{R}^3)}.
		\end{equation*}
		Note that $E_1(R)$ is a complete space with \(d_1\) in the first case and \(d_2\) in the second. Let $R:=C\|\U_0\|_{\dot{H}_x^s(\R^3)}$. We restrict the variable $(t,x)\in I\times \R^3$. By the Strichartz estimate, for any admissible pair $(q,r)$,
		\begin{equation*}
			\begin{aligned}
				\||\nabla|^s \Phi(\U)\|_{L_t^{q}L_x^r\cap L_t^{q_0}L_x^{\wt{r}_0}}&\leq C\| \U_0\|_{\dot{H}_x^s(\R^3)} +C\|t^{\frac{3p}{2}-2}\|_{L_t^{\frac{4}{4-(3-2s)p}}(I)} \||\nabla|^s \U\|^{p+1}_{L_t^{q_0}L_x^{\wt{r}_0}}\\
				&\leq R+CT_0^{\frac{(3+2s)p-4}{4}} R^{p+1}.
			\end{aligned}
		\end{equation*}
		Similarly, by Lemma \ref{2: Lemma-H^s,0<s<1},
		\begin{equation*}
			\begin{aligned}
				d_1(\Phi(\U), \Phi(\V))&\leq CT_0^{\frac{(3+2s)p-4}{4}}(\||\nabla|^s \U\|^p_{L_t^{q_0}L_x^{\wt{r}_0}} + \||\nabla|^s \V\|^p_{L_t^{q_0}L_x^{\wt{r}_0}})d_1(\U, \V)\\
				&\leq CT_0^{\frac{(3+2s)p-4}{4}} R^pd_1(\U,\V).
			\end{aligned}
		\end{equation*}
		By Lemma \ref{lem: Nonlinear estimate},
		\begin{equation*}
			\begin{aligned}
				d_2(\Phi(\U),\Phi(\V))&\leq CT_0^{\frac{(3+2s)p-4}{4}}(\||\nabla|^s \U\|^p_{L_t^{q_0}L_x^{\wt{r}_0}} + \||\nabla|^s \V\|^p_{L_t^{q_0}L_x^{\wt{r}_0}}) d_2(\U,\V)\\
				&\leq CT_0^{\frac{(3+2s)p-4}{4}}R^p d_2(\U,\V).
			\end{aligned}
		\end{equation*}
		Choose $T_0 =T_0(\|\U_0\|_{\dot{H}_x^s(\R^3)})$ sufficiently small such that the map $\Phi(\U)$ is a contraction. Then we obtain a unique solution $\U$ satisfying $|\nabla|^s \U\in L_t^{q}L_x^{r}(I\times \mathbb{R}^3)$ for any admissible pair $(q,r)$. Therefore, $\U\in C(I;\dot{H}^s(\mathbb{R}^3))$.
		
		Next, we consider $p=\frac{4}{3+2s}$. We denote the resolution space
		\begin{equation*}
			E_2(\delta):=\{ \U\in L_t^{q_0,2}\dot{W}_x^{s,\wt{r}_0}(I\times \mathbb{R}^3) :\||\nabla|^s \U\|_{L_t^{q_0,2}L_x^{\wt{r}_0}(I\times \mathbb{R}^3)}\leqslant 2\delta \}.
		\end{equation*}
		When $p\geq 1$ we consider the metric
		\begin{equation*}
			d_3(\U,\V)=\||\nabla|^{s} (\U-\V)\|_{L_t^{q_0,2}L_x^{\wt{r}_0}(I\times \mathbb{R}^3)}.
		\end{equation*} 
		When $p< 1$ we consider the metric
		\begin{equation*}
			d_4(\U,\V)=\|\U-\V\|_{X_2(I)}=\||\nabla|^{\theta} (\U-\V)\|_{L_t^{q_0,2}L_x^{r_0}(I\times \mathbb{R}^3)}.
		\end{equation*} 
		By the Strichartz estimate,
		\begin{equation*}
			\begin{aligned}
				\||\nabla|^s \Phi(\U)\|_{ L_t^{q_0,2}L_x^{\wt{r}_0}}&\leq C\||\nabla|^s S(t)\U_0\|_{L_t^{q_0,2}L_x^{\wt{r}_0}} +\||\nabla|^s \U\|^{p+1}_{L_t^{q_0,2}L_x^{\wt{r}_0}}\\
				&\leq C\||\nabla|^s S(t)\U_0\|_{L_t^{q_0,2} L_x^{\wt{r}_0}} +C\delta^{p+1}.
			\end{aligned}
		\end{equation*}
		Similarly, for any admissible pair $(q,r)$, we have
		\begin{equation*}
			\||\nabla|^s \Phi(\U)\|_{ L_t^{q,2}L_x^{r}}\leq C\| \U_0\|_{\dot H_x^s(\R^3)} +C\||\nabla|^s \U\|^{p+1}_{L_t^{q_0,2}L_x^{\wt{r}_0}}.
		\end{equation*}
		By Lemmas \ref{2: Lemma-H^s,0<s<1} and \ref{lem: Nonlinear estimate},
		\begin{equation*}
			\begin{aligned}
				d_3(\Phi(\U),\Phi(\V))&\leq C\delta^p d_3(\U,\V),\\
				d_4(\Phi(\U),\Phi(\V))&\leq C\delta^p d_4(\U,\V).
			\end{aligned}
		\end{equation*}
		Choose $\delta$ and $T_0=T_0(\U_0)$ sufficiently small so that $C\||\nabla|^s S(t)\U_0\|_{L_t^{q_0,2} L_x^{\wt{r}_0}(I\times\R^3)}\leq \delta$. Then the map $\Phi(\U)$ is a contraction. Thus, we obtain a unique solution $\U\in C(I;\dot{H}^s(\mathbb{R}^3))$.
		
		It remains to prove the asserted stability estimates for the solution map. By the same smallness condition used in the contraction argument, the solution map is defined on a sufficiently small neighborhood of \(\U_0\) with a common existence time \(T_0\). Let \(\U_{0,1}\) and \(\U_{0,2}\) be two initial data in this neighborhood, and let \(\U_1\) and \(\U_2\) be the corresponding solutions in \(C([0,T_0];\dot H^s(\mathbb R^3))\). Set
		\begin{equation*}
			\eta:=\|\U_{0,1}- \U_{0,2}\|_{\dot H_x^s(\R^3)}\ll 1.
		\end{equation*}
		
		When \(p\ge1\), the difference estimate used in the contraction
		argument, together with the Strichartz estimate, gives
		\[
		\|\U_1-\U_2\|_{C([0,T_0];\dot H_x^s(\R^3))}
		\lesssim \eta.
		\]
		
		When \(0<p<1\), the linear estimate, the Duhamel formula, and
		Lemma \ref{lem: Nonlinear estimate} give
		\[
		d_j(\U_1,\U_2)\lesssim \eta,
		\]
		where \(j=2\) in the non-endpoint case and \(j=4\) in the endpoint
		case. Applying the Strichartz estimate and
		Lemma \ref{2: Lemma-H^s,0<s<1} at regularity \(s\), and absorbing the
		term containing the highest-order difference norm by the same
		smallness condition used in the contraction argument, we obtain
		\[
		\|\U_1-\U_2\|_{C([0,T_0];\dot H_x^s(\R^3))}
		\lesssim 
		\eta+d_j(\U_1,\U_2)^p
		\lesssim \eta^p.
		\]
		This proves the claimed stability estimates.
	\end{proof}

	\subsection{Ill-posedness}\label{section-illposed}

	Applying the pseudo-conformal transform, Theorem \ref{thm:final-local} (2) is equivalent to 
	\begin{prop}\label{prop-illposed-initial}
		Let the assumptions in Theorem \ref{thm:final-local} (2) hold. Then for any $0<\varepsilon, \delta<1$ and for any $t>0$ sufficiently small there exist solutions $\U_1$, $\U_2$ of \eqref{eq:nls-initial} with initial data $\U_1(0)$, $\U_2(0) \in \mathcal{S}(\mathbb{R}^3)$ such that
		\begin{align*}
			\|\U_1(0)\|_{\dot{H}_x^s(\R^3)}+\|\U_2(0)\|_{\dot{H}_x^s(\R^3)} &< C \varepsilon,\\
			\|\U_1(0)-\U_2(0)\|_{\dot{H}_x^s(\R^3)}&<C \delta,\\
			\|\U_1(t)-\U_2(t)\|_{\dot{H}_x^s(\R^3)}&> c \varepsilon,
		\end{align*}
		where $C,c>0$ are constants independent of $\varepsilon$ and $\delta$.
	\end{prop}
	
	In the following, we prove Proposition \ref{prop-illposed-initial}. Note that equation \eqref{eq:nls-initial} is invariant under the scaling transform:
	\begin{equation}\label{scale transform}
		\U(t,x) \rightarrow \lambda^{\frac{2}{p}-3}\U(\lambda^{-2}t, \lambda^{-1}x).
	\end{equation}
	Denote
	\begin{equation}\label{illposedness-phi transform}
		\U(t,x)=\phi(t,\mu x),
	\end{equation}
	where $\U(t,x)$ is the solution of \eqref{eq:nls-initial}. Then $\phi$ satisfies the equation
	\begin{align}\label{NLS: phi equation}
		\left\{
		\begin{aligned}
			&i\phi_t+\frac{\mu^2}{2}\Delta \phi=t^{\frac{3p}{2}-2}|\phi|^p \phi, \\
			&\phi(0,x)=\phi_0(x).
		\end{aligned}
		\right.
	\end{align}
	Consider the ODE
	\begin{align}\label{ODE-equation}
		\left\{
		\begin{aligned}
			&i\phi_t=t^{\frac{3p}{2}-2}|\phi|^p \phi, \\
			&\phi(0,x)=\phi_0(x),
		\end{aligned}
		\right.
	\end{align}
	which has the explicit solution $\phi^{(0)}$ defined by
	\begin{equation}
		\phi^{(0)}(t,x)=\phi_0(x)e^{-i\frac{2}{3p-2}t^{\frac{3p}{2}-1}|\phi_0(x)|^p}, \quad \text{for} \ p> \frac{2}{3}.
	\end{equation}
	Moreover, one may choose a suitable function $\phi_0(x)$ such that for every $t\geq 0$, $\phi^{(0)}(t,x) \in \mathcal{S}(\R^3)$. For instance, taking $\phi_0(x) =e^{-|x|^2}\in \mathcal{S}(\R^3)$ suffices.

	We need the following quantitative result.
	\begin{lem}\label{Lem1-quantitative}
		Let $\frac{2}{3}<p\leq 1$. Choose $\phi_0\in\mathcal{S}(\mathbb{R}^3)$ such that $\phi^{(0)}\in \mathcal{S}(\R^3)$. Then there exist constants $C,c>0$ depending only on \(p\) and \(\phi_0\), such that if $0<\mu\leq c$ is a sufficiently small real number, then for $T=c|\log \mu|^{c}$ there exists a solution $\phi\in C([0,T];L_x^2(\R^3))$ of \eqref{NLS: phi equation} satisfying 
		\begin{equation*}
			\sup_{0\leq t\leq c|\log \mu|^{c}} \|\phi(t)-\phi^{(0)}(t)\|_{L_x^2(\R^3)}\leq C\mu.
		\end{equation*}
	\end{lem}
	\begin{proof}
		Since $\phi_0$ is a Schwartz function, the existence of $\phi$ follows from \cite{Nak01siam}. We define $F(z):=|z|^pz$ and 
		\begin{equation*}
			w:=\phi-\phi^{(0)}.
		\end{equation*}
		Then $w$ solves the Cauchy problem
		\begin{align}\label{Lem: Cauchy problem}
			\left\{
			\begin{aligned}
				&iw_t+\frac{\mu^2}{2}\Delta w=-\frac{\mu^2}{2}\Delta\phi^{(0)}+t^{\frac{3p}{2}-2}(F(\phi^{(0)}+w)-F(\phi^{(0)})), \\
				&w(0,x)=0.
			\end{aligned}
			\right.
		\end{align}
		It suffices to show that 
		\begin{equation}\label{Lem- bound}
			\sup_{0\leq t\leq T}\|w(t)\|_{L_x^2(\R^3)}\leq C\mu,
		\end{equation} 
		where $0\leq T\leq c|\log\mu|^c$. 
		
		Taking the \(L^2\)-inner product of \eqref{Lem: Cauchy problem} with \(w\) and then taking imaginary parts, we obtain
		\begin{equation}\label{Lem-derivate}
			\begin{aligned}
				&\frac{\mathrm{d}}{\mathrm{d}t}\|w(t)\|^2_{L_x^2(\R^3)}=2\re \int_{\mathbb{R}^3}w_t\bar{w}\mathrm{d}x\\
				&=-2 \im\int_{\mathbb{R}^3}\big(\frac{\mu^2}{2}\Delta w+\frac{\mu^2}{2}\Delta \phi^{(0)}-t^{\frac{3p}{2}-2}(F(\phi^{(0)}+w)-F(\phi^{(0)}))\big)\bar{w}\mathrm{d}x\\
				&\leq C \mu^2 \|w(t)\|_{L_x^2(\R^3)} \|\Delta\phi^{(0)}\|_{L_x^2(\R^3)} +Ct^{\frac{3p}{2}-2}\left|\im\int_{\mathbb{R}^3}(F(\phi^{(0)}+w)-F(\phi^{(0)}))\bar{w}\mathrm{d}x\right|.
			\end{aligned}
		\end{equation}
		Denote
		\begin{equation*}
			A(t):=\{ x\in\mathbb{R}^3: |w(t,x)|\leq |\phi^{(0)}(t,x)| \}, \quad A^{c}(t):=\{ x \in \R^3: x \notin A(t) \}.
		\end{equation*}
		Then,
		\begin{equation}\label{Lem: Im F-1}
			\begin{aligned}
				&\left|\im\int_{A(t)}(F(\phi^{(0)}+w)-F(\phi^{(0)}))\bar{w}\mathrm{d}x\right| \\
				&\leq C \int_{A(t)}|w|(|\phi^{(0)}|^p+|w|^p)|\bar{w}|\mathrm{d}x\\
				&\leq C  \|w(t)\|_{L_x^2(\R^3)}\|\chi_{A(t)}|\phi^{(0)}(t)|^pw(t)\|_{L_x^2(\R^3)},
			\end{aligned}
		\end{equation}
		and 
		\begin{equation}\label{Lem: Im F-2}
			\begin{aligned}
				&\left|\im\int_{A^c(t)}(F(\phi^{(0)}+w)-F(\phi^{(0)}))\bar{w}\mathrm{d}x\right|\\
				&=\left|\im\int_{A^c(t)}(|\phi^{(0)}+w|^p-|\phi^{(0)}|^p)\phi^{(0)}\bar{w}\mathrm{d}x\right|\\
				&\leq \|w(t)\|_{L_x^2(\R^3)}\|\chi_{A^{c}(t)}|w(t)|^p\phi^{(0)}(t)\|_{L_x^2(\R^3)}.
			\end{aligned}
		\end{equation}
		Since $p\leq 1$ and $\phi_0$ is Schwartz, \eqref{Lem: Im F-1} and \eqref{Lem: Im F-2} imply that 
		\begin{equation}\label{Lem-Im F}
			\left|\im\int_{\mathbb{R}^3}(F(\phi^{(0)}+w)-F(\phi^{(0)}))\bar{w}\mathrm{d}x\right|\leq C\|w(t)\|^2_{L_x^2(\R^3)}.
		\end{equation}
		By \eqref{Lem-derivate} and \eqref{Lem-Im F},
		\begin{equation*}
			\frac{\mathrm{d}}{\mathrm{d}t}\|w(t)\|_{L_x^2(\R^3)}\leq C\mu^2(1+|t|)^C+ Ct^{\frac{3p}{2}-2}\|w(t)\|_{L_x^2(\R^3)}.
		\end{equation*}
		By Gronwall's inequality and the initial condition $w(0)=0$, we obtain 
		\begin{equation*}
			\|w(t)\|_{L_x^2(\R^3)}\leq C\mu^2e^{C(1+|t|)^C}.
		\end{equation*}
		Thus if $0\leq t<c|\log\mu|^c$ for suitably chosen $c$ and $\mu$ is sufficiently small, we obtain \eqref{Lem- bound}.
	\end{proof}
	By Lemma \ref{Lem1-quantitative}, it follows that for $\mu\leq c$ there exists a solution $\phi^{(a,\mu)}(t,x)$ to the equation \eqref{NLS: phi equation} with initial datum 
	\begin{equation*}
		\phi^{(a,\mu)}(0,x)=aw(x),
	\end{equation*}
	where $a\in[\frac{1}{2},1]$ and $w(x)\in \mathcal{S}(\mathbb{R}^3)$ is a fixed nonzero function chosen so that \(\phi^{(a,0)}(t)\in\mathcal{S}(\R^3)\) for every $a\in[\frac{1}{2},1]$ and $t\geq 0$. Moreover, we have
	\begin{equation}\label{dis-continuous estimate}
		\sup_{0\leq t\leq c|\log \mu|^{c}} \|\phi^{(a,\mu)}(t)-\phi^{(a,0)}(t)\|_{L_x^2(\R^3)}\leq C\mu,
	\end{equation}
	where $\phi^{(a,0)}(t,x)$ is the solution of \eqref{ODE-equation}, 
	\begin{equation}\label{ill-initial data}
		\phi^{(a,0)}(t,x)=aw(x)e^{-i\frac{2}{3p-2}t^{\frac{3p}{2}-1}a^p|w(x)|^p}.
	\end{equation}
	
	By \eqref{scale transform} and \eqref{illposedness-phi transform}, we obtain a three-parameter family of solutions $\U^{(a,\mu,\lambda)}$ to the NLS equation \eqref{eq:nls-initial}, where $\frac{1}{2}\leq a\leq 1, 0<\lambda < \mu \ll 1$, and 
	\begin{equation*}
		\U^{(a,\mu,\lambda)}(t,x)=\lambda^{\frac{2}{p}-3} \phi^{(a,\mu)}(\lambda^{-2}t,\lambda^{-1}\mu x).
	\end{equation*}
	We prove the proposition by choosing the three parameters appropriately.
	\begin{proof}[Proof of Proposition \ref{prop-illposed-initial}]
		By a simple calculation,
		\begin{equation}\label{illpose-result-1}
			\begin{aligned}
				\|\U^{(a,\mu,\lambda)}(0,x)\|_{\dot{H}_x^s(\R^3)}&=a\lambda^{-s_c-s}\mu^{s-\frac{3}{2}}\|w(x)\|_{\dot{H}_x^s(\R^3)},\\
				\|\U^{(a,\mu,\lambda)}(0,x)-\U^{(a^{\prime},\mu,\lambda)}(0,x)\|_{\dot{H}_x^s(\R^3)}&=|a-a^{\prime}|\lambda^{-s_c-s}\mu^{s-\frac{3}{2}}\|w(x)\|_{\dot{H}_x^s(\R^3)}.
			\end{aligned}
		\end{equation}
		Let $0<\varepsilon<1$ be fixed. Since $p<\frac{4}{3+2s}$, we have $-s_c-s>0$. We set
		\begin{equation*}
			\lambda^{-s_c-s}\mu^{s-\frac{3}{2}}=\varepsilon.
		\end{equation*}
		On the other hand,
		\begin{equation}\label{proof-illposed-3}
			\begin{aligned}
				\|\U^{(a,\mu,\lambda)}-\U^{(a^{\prime}, \mu, \lambda)}\|^2_{\dot{H}_x^s(\R^3)}	&=\varepsilon^2 \int_{\mathbb{R}^3}|\xi|^{2s}|\widehat{\phi^{(a,\mu)}(\lambda^{-2}t)}(\xi) - \widehat{\phi^{(a^{\prime},\mu)}(\lambda^{-2}t)}(\xi)|^2\mathrm{d}\xi\\
				&\geq \varepsilon^2 \int_{|\xi|\geq k}|\xi|^{2s} |\widehat{\phi^{(a,\mu)}}(\lambda^{-2}t)-\widehat{\phi^{(a^{\prime},\mu)}}(\lambda^{-2}t)|^2\mathrm{d}\xi\\
				&\geq \frac{1}{2}\varepsilon^2 k^{2s}\int_{|\xi|\geq k} |\widehat{\phi^{(a,0)}}(\lambda^{-2}t) - \widehat{\phi^{(a^{\prime},0)}}(\lambda^{-2}t)|^2\mathrm{d}\xi \\
				&\qquad - 2\varepsilon^2 k^{2s}\int_{|\xi|\geq k} |\widehat{\phi^{(a,\mu)}}(\lambda^{-2}t) - \widehat{\phi^{(a,0)}}(\lambda^{-2}t)|^2\mathrm{d}\xi \\
				&\qquad - 2\varepsilon^2 k^{2s}\int_{|\xi|\geq k} |\widehat{\phi^{(a^{\prime},\mu)}}(\lambda^{-2}t) - \widehat{\phi^{(a^{\prime},0)}}(\lambda^{-2}t)|^2\mathrm{d}\xi,
			\end{aligned}
		\end{equation}
		where $k>0$ is defined later. From an inspection of \eqref{ill-initial data} we see that there exists $T=T(a,a^{\prime})>0$ such that 
		\begin{equation*}
			\|\widehat{\phi^{(a,0)}}(T)-\widehat{\phi^{(a^{\prime},0)}}(T)\|^2_{L_{\xi}^2(\R^3)}=\|\phi^{(a,0)}(T)-\phi^{(a^{\prime},0)}(T)\|^2_{L_x^2(\R^3)}\geq c_1,
		\end{equation*}
		where $c_1>0$ is independent of $a$ and $a^{\prime}$. Note that $\|\widehat{\phi^{(a,0)}}\|_{L_{\xi}^{\infty}(\R^3)}\leq \|\phi^{(a,0)}\|_{L_x^1(\R^3)}\leq C$,
		\begin{equation}\label{proof-illposed-3-I}
			\begin{aligned}
				&\int_{|\xi|\geq k} |\widehat{\phi^{(a,0)}}(T) - \widehat{\phi^{(a^{\prime},0)}}(T)|^2\mathrm{d}\xi \\
				&= \|\phi^{(a,0)}(T) - \phi^{(a^{\prime},0)}(T)\|^2_{L_x^2(\R^3)}- \int_{|\xi|< k}|\widehat{\phi^{(a,0)}}(T) - \widehat{\phi^{(a^{\prime},0)}}(T)|^2\mathrm{d}\xi\\
				&\geq c_1 - c_2k^3,
			\end{aligned}
		\end{equation}
		where $c_2>0$ is independent of $a$ and $a^{\prime}$.
		
		Thus, set $k=(\frac{c_1}{2c_2})^{\frac{1}{3}}$ and let $t=\lambda^2 T$. Then, by \eqref{proof-illposed-3}, \eqref{proof-illposed-3-I} and \eqref{dis-continuous estimate}, we obtain
		\begin{equation}\label{illpose-result-2}
			\|\U^{(a,\mu,\lambda)}(\lambda^2 T)-\U^{(a^{\prime}, \mu, \lambda)}(\lambda^2T)\|^2_{\dot{H}_x^s(\R^3)}\geq \frac{c_1}{4}k^{2s}\varepsilon^2-Ck^{2s}\mu^2\varepsilon^2.
		\end{equation}
		
		Choose $a\neq a'$ such that $|a-a^{\prime}|<\delta$. Let $\mu\rightarrow 0$, and hence $\lambda^2 T\rightarrow 0$. By \eqref{illpose-result-1} and \eqref{illpose-result-2}, we finish the proof of Proposition \ref{prop-illposed-initial}. 
	\end{proof}

	\vspace{2cm}
	
	\section{Finite pseudo-conformal energy}\label{sec:finite pce}
	
	\vspace{0.5cm}

	We are now in a position to prove Proposition \ref{thm:gwp}. After employing the pseudo-conformal transform, it suffices to prove that the equation \eqref{eq:nls-initial} admits a solution $\U\in C([0,T_0];H_x^1(\R^3))$ for any $T_0>0$. Recall the equation
	\begin{equation}\label{eq:nls-h1}
		\U(t)=S(t)\U_0-i\int_{0}^t S(t-\tau) \ta^{\frac{3p}2-2}|\U(\ta)|^p\U(\ta)\dd \ta,
	\end{equation}
	with $\U_0\in H_x^1(\R^3)$. We first give a local result in $H_x^1(\R^3)$.
	\begin{prop}\label{prop:h1gwp-lwpI}
		Let $\frac{4}{3}< p \leq 4$. For any $\U_0\in H_x^1(\R^3)$, there exists $t_0=t_0(\norm{\U_0}_{H_x^1(\R^3)})$ such that the equation \eqref{eq:nls-h1} admits a unique solution
		\begin{equation*}
			\U\in C([0,t_0];H_x^1(\R^3)).
		\end{equation*}
	\end{prop}
	\begin{proof}
		Let $I:=[0,t_0]$ for some $t_0>0$. We restrict the variable $(t,x)\in I\times\R^3$. By the Strichartz estimate, we have that for any admissible pair $(q,r)$,
		\EQ{
			\normB{\jb{\nabla}\int_0^t S(t-\tau) \ta^{\frac{3p}2-2}|\U(\ta)|^{p}\U(\ta)\dd \ta}_{L_t^{q} L_x^r} \lsm & \normb{\ta^{\frac{3p}2-2}|\U|^{p}\jb{\nabla}\U}_{L_t^{\frac{8}{8-p}} L_x^{\frac{12}{p+6}}} \\
			\lsm & \normb{t^{\frac{3p}2-2}}_{L_t^{\frac{4}{4-p}}} \norm{\U}_{L_t^8 L_x^{12}}^p \normb{\jb{\nabla}\U}_{L_t^{\I} L_x^{2}} \\
			\lsm & t_0^{\frac54p-1} \norm{\nabla\U}_{L_t^{8} L_x^{\frac{12}{5}}}^p \norm{\jb{\nabla}\U}_{L_t^{\I} L_x^{2}},
		}
		where $\frac{4}{4-p}:=\infty$ when $p=4$. Then, this proposition follows by the standard contraction mapping argument in
		\begin{equation*}
			\fbrk{\U\in C(I; H_x^{1}(\R^3)): \norm{\jb{\nabla}\U}_{L_t^{\I} L_x^{2} \cap L_t^{8} L_x^{\frac{12}{5}}(I\times\R^3)} \le R},
		\end{equation*}
		where $R:=C\norm{\U_0}_{H_x^1(\R^3)}$ and $t_0$ denotes some constant depending on $\norm{\U_0}_{H_x^1(\R^3)}$.
	\end{proof}
	
	\subsection{A priori estimate}
	
	For \eqref{eq:nls-initial}, we define the pseudo-conformal energy
	\begin{equation}\label{Pseudo-conformal-energy}
		\wt{\E}(t):=\frac{1}{4}t^{2-\frac{3p}{2}}\int_{\R^3}|\nabla\U|^2\mathrm{d}x+\frac{1}{p+2}\int_{\R^3}|\U|^{p+2}\mathrm{d}x.
	\end{equation}
	\begin{prop}[A priori estimate]\label{prop:h1gwp-energy}
		Let $\frac{4}{3}<p\leq 4$. Assume $\U_0\in H_x^1(\R^3)$ and that, for some $t_3\in [t_0,T_0]$, $\U\in C([t_0,t_3]; H_x^1(\R^3))$, where $t_0$ is given in Proposition \ref{prop:h1gwp-lwpI}. Then, we have
		\begin{equation}\label{esti-E(t)}
			\sup_{t\in[t_0,t_3]}\wt \E(t)\le C(\norm{\U_0}_{ H_x^1(\R^3)}).
		\end{equation}
		Moreover, 
		\begin{equation}\label{esti-H1}
			\sup_{t_0\leq t\leq t_3} \|\U(t)\|^2_{H_x^1(\R^3)} \leq C(\|\U_0\|_{H_x^1(\R^3)})\big(T_0^{\frac{3p}{2}-2}+1\big).
		\end{equation}
	\end{prop}
	\begin{proof}
		By Sobolev's inequality and Proposition \ref{prop:h1gwp-lwpI},
		\begin{equation*}
			\wt \E(t_0) \lsm_{t_0}  \|\U\|_{L_t^\infty([0,t_0];H_x^1(\R^3))}^2 + \|\U\|_{L_t^\infty ([0,t_0];H_x^1(\R^3))}^{p+2}
			\le C(\|\U_0\|_{H_x^1(\R^3)}).
		\end{equation*}
		
		Recall the equation for $\U$:
		\begin{equation*}
			i\pd_t \U + \frac12\De \U = t^{\frac{3p}{2}-2} |\U|^p\U.
		\end{equation*}
		Multiplying the equation by $t^{2-\frac{3p}{2}}\wb{\U_t}$, integrating in $x$, and taking the real part, we obtain 
		\begin{equation*}
			\frac12t^{2-\frac {3p}2}\re\int_{\R^3} \De \U\wb\U_t \dd x = \re\int_{\R^3} |\U|^p\U \wb\U_t \dd x.
		\end{equation*}
		Integrating by parts, we obtain
		\begin{equation*}
			- \frac12t^{2-\frac {3p}2}\re\int_{\R^3} \nabla \U\cdot\nabla\wb\U_t \dd x = \re\int_{\R^3} |\U|^p\U \wb\U_t \dd x,
		\end{equation*}
		which gives
		\begin{equation*}
			- \frac14t^{2-\frac {3p}2}\pd_t\brkb{\int_{\R^3} |\nabla \U|^2 \dd x} = \frac{1}{p+2} \pd_t\brkb{\int_{\R^3} |\U|^{p+2} \dd x}.
		\end{equation*}
		Since $p>\frac43$, we have
		\begin{equation*}
			\pd_t\wt \E(t) = \frac14(2-\frac {3p}2)t^{1-\frac {3p}2} \int_{\R^3} |\nabla \U|^2 \dd x 
			\le  0.
		\end{equation*}
		Therefore, for any $t_0\le t\le t_3$,
		\begin{equation*}
			\wt \E(t) \le \wt \E(t_0) \le C(\norm{\U_0}_{ H_x^1(\R^3)}).
		\end{equation*}
		Thus, \eqref{esti-E(t)} holds. By \eqref{Pseudo-conformal-energy}, we have
		\begin{equation*}
			\sup_{t_0\leq t\leq t_3}\|\nabla\U(t)\|^2_{L_x^2(\R^3)} \leq C(\norm{\U_0}_{ H_x^1(\R^3)}) T_0^{\frac{3p}{2}-2}. 
		\end{equation*}
		By conservation of mass, we obtain \eqref{esti-H1}.
	\end{proof}
	
	\subsection{Extension of the solution}
	In order to extend the solution, consider the equation
	\begin{equation}\label{eq:nls-h1gwp}
		\U(t)=S(t-t_1)\U(t_1)-i\int_{t_1}^t S(t-\tau) \ta^{\frac{3p}2-2}|\U(\ta)|^p\U(\ta)\dd \ta.
	\end{equation}
	
	\begin{prop}\label{prop:h1gwp-lwpII}
		Let $\frac{4}{3}<p<4$, $\U(t_1,x)\in H_x^{1}(\R^3)$, $T_0>1$, and $t_0$ be given in Proposition \ref{prop:h1gwp-lwpI}. Then, for any $t_1\in [t_0,T_0)$, there exists $0<t_2\le\min\fbrk{1,T_0-t_1}$ depending on $\norm{\U(t_1)}_{H^1(\R^3)}$ and $T_0$ such that the equation \eqref{eq:nls-h1gwp} admits a unique solution $\U \in C([t_1,t_1+t_2]; H_x^1(\R^3))$.
	\end{prop}
	\begin{proof}
		Take some $t_2$ that will be defined later such that $0<t_2\le\min\fbrk{1,T_0-t_1}$. In the proof of this proposition, we denote $I:=[t_1,t_1+t_2]$, and restrict $(t,x)$ to $I\times\R^3$. We have that for any admissible $(q,r)$,
		\EQ{
			\normB{\jb{\nabla}\int_{t_1}^t S(t-\tau) \ta^{\frac{3p}2-2}|\U(\ta)|^{p}\U(\ta)\dd \ta}_{L_t^{q} L_x^r} & \lsm  \normb{t^{\frac{3p}2-2}|\U|^{p}\jb{\nabla}\U}_{L_t^{\frac{8}{8-p}} L_x^{\frac{12}{p+6}}} \\
			& \lsm T_0^{\frac{3p}2-2} t_2^{\frac{4-p}{4}} \norm{\U}_{L_t^8 L_x^{12}}^p \normb{\jb{\nabla}\U}_{L_t^{\I} L_x^{2}} \\
			&\lsm_{T_0}  t_2^{\frac{4-p}{4}} \norm{\nabla\U}_{L_t^{8} L_x^{\frac{12}{5}}}^p \norm{\jb{\nabla}\U}_{L_t^{\I} L_x^{2}}.
		}
		Then, this proposition follows by the standard contraction mapping argument in the resolution space
		\begin{equation*}
			\fbrk{\U\in C(I; H_x^{1}(\R^3)): \norm{\jb{\nabla}\U}_{L_t^{\I} L_x^{2} \cap L_t^{8} L_x^{\frac{12}{5}} (I\times\R^3)}\le R},
		\end{equation*}
		where $R:=C\norm{\U(t_1)}_{H_x^1(\R^3)}$.
	\end{proof}

	The endpoint \(p=4\) is more delicate, since the time-smallness factor in Proposition \ref{prop:h1gwp-lwpII} disappears. We handle this case using the energy-critical NLS theory and a perturbative argument.
	
	\begin{prop}\label{prop:p4-endpoint}
		Let $p=4$, $\U(t_1,x)\in H_x^1(\mathbb R^3)$, $T_0>1$, and $t_0$ be given in Proposition~\ref{prop:h1gwp-lwpI}. Then, for any $t_1\in[t_0,T_0)$, there exists $0<t_2\leq\min\{1,T_0-t_1\},$
		depending on $\|\U(t_1)\|_{H_x^1(\mathbb R^3)}$, $t_0$, and $T_0$, such that equation \eqref{eq:nls-h1gwp} admits a unique solution $\U\in C\bigl([t_1,t_1+t_2];H_x^1(\mathbb R^3)\bigr).$
	\end{prop}
	
	\begin{proof}	
		We denote $I:=[t_1, t_1+t_2]$. Consider first the frozen-coefficient equation
		\begin{equation}\label{eq:p4-frozen-equation}
			\V(t)=S(t-t_1)\U(t_1)
			-i t_1^4\int_{t_1}^tS(t-\tau)|\V(\tau)|^4\V(\tau)\,d\tau.
		\end{equation}
		Setting $\W=t_1\V$, we obtain
		\begin{equation*}
			i\partial_t\W+\frac12\Delta \W=|\W|^4\W,
			\qquad \W(t_1)=t_1\U(t_1).
		\end{equation*}
		This is the standard three-dimensional defocusing energy-critical NLS. Define the conserved energy by
		\begin{equation*}
			E(\W(t)):=\frac14\|\nabla \W(t)\|^2_{L_x^2(\mathbb R^3)}+\frac16\|\W(t)\|^6_{L_x^6(\R^3)}.
		\end{equation*}
		Since $\W(t_1)=t_1\U(t_1)$ and $t_1\leq T_0$, Sobolev embedding gives
		\begin{equation*}
			E(\W(t_1))\lesssim t_1^2\|\U(t_1)\|_{H^1_x(\R^3)}^2+t_1^6\|\U(t_1)\|_{H^1_x(\R^3)}^6
			\lesssim C(T_0, \|\U(t_1)\|_{H^1_x(\R^3)}).
		\end{equation*}
		Hence, by \cite[Theorem~1.1 and Lemma~3.12]{CKSTT08Annals}, $\W$ is global and satisfies
		\begin{equation}\label{eq:p4-W-bound}
			\|\W\|_{L_{t,x}^{10}(\mathbb R\times\mathbb R^3)}
			+\|\nabla \W\|_{L_t^8L_x^{12/5}(\mathbb R\times\mathbb R^3)}
			+\|\nabla \W\|_{L_t^\infty L_x^2(\mathbb R\times\mathbb R^3)}
			\leq B_0,
		\end{equation}
		where $B_0=B_0(\|\U(t_1)\|_{H_x^1(\R^3)},T_0)$. By mass conservation, we have $\|\W\|_{L_t^\infty L_x^2(\mathbb R\times\mathbb R^3)}
		\leq t_1\|\U(t_1)\|_{H_x^1(\R^3)}$. Enlarging $B_0$ if necessary, and using $\V=t_1^{-1}\W$ and $t_1\geq t_0$, we obtain
		\begin{equation}\label{eq:p4-frozen-bound}
			\|\V\|_{L_{t,x}^{10}(\mathbb R\times\mathbb R^3)}
			+\|\nabla \V\|_{L_t^8L_x^{12/5}(\mathbb R\times\mathbb R^3)}
			+\|\langle \nabla \rangle \V\|_{L_t^\infty L_x^2(\mathbb R\times\mathbb R^3)}\leq t_0^{-1}B_0.
		\end{equation}
		
		We regard $\V$ as an approximate solution to equation \eqref{eq:nls-h1gwp}. Its error is
		\begin{equation*}
			e(t,x)=(t^4-t_1^4)|\V|^4\V.
		\end{equation*}
		Since $t\in I\subset[t_0,T_0]$, we have
		\begin{equation*}
			\sup_{t\in I}|t^4-t_1^4|\lesssim T_0^3t_2.
		\end{equation*}
		Therefore, by Sobolev embedding and \eqref{eq:p4-frozen-bound},
		\begin{equation}
			\|\langle \nabla \rangle e\|_{L_t^2L_x^{\frac{6}{5}}(I\times\mathbb R^3)}\lesssim T_0^3t_2
			\|\V\|_{L_t^8L_x^{12}(I\times\mathbb R^3)}^4
			\|\langle \nabla \rangle \V\|_{L_t^\infty L_x^2(I\times\mathbb R^3)} \lesssim t_0^{-5} T_0^3t_2B_0^5.\label{eq:p4-error-bound}
		\end{equation}
		
		We now give the perturbative argument more explicitly. Let
		\[
		Z:=\U-\V,\qquad F(z):=|z|^4z.
		\]
		Then $Z(t_1)=0$ and
		\begin{equation}\label{eq:p4-difference-equation}
			Z(t)=-i\int_{t_1}^tS(t-\tau)
			\left\{\tau^4\bigl[F(\V+Z)-F(\V)\bigr]
			+e(\tau, x) \right\} \mathrm{d}\tau.
		\end{equation}
		Fix a sufficiently small constant $\eta>0$. By \eqref{eq:p4-frozen-bound}, we can divide $I$ into finitely many consecutive subintervals $I=I_1\cup\cdots\cup I_N$ ($I_j:=[a_j, a_{j+1}]$) such that
		\begin{equation*}
			\|\nabla \V\|_{L_t^8L_x^{12/5}(I_j\times\mathbb R^3)}\leq\eta\ll1,
			\qquad 1\leq j\leq N,
		\end{equation*}
		where $N$ depends only on $B_0$, $t_0$ and $\eta$. In the following, we restrict $(t,x)$ to $I_j\times\R^3$. For any admissible $(q,r)$, by the Strichartz estimate and \eqref{eq:p4-error-bound},
		\begin{equation*}
			\|\langle \nabla \rangle Z\|_{L_t^qL_x^r}\lsm_{t_0, T_0} \|Z(a_j)\|_{H_x^1} + \big(  \|\nabla\V\|_{L_t^{8} L_x^{\frac{12}{5}}}^4 +  \|\nabla Z\|_{L_t^{8} L_x^{\frac{12}{5}}}^4\big) \norm{\jb{\nabla}Z}_{L_t^{\I} L_x^{2}} + t_2B_0^5.
		\end{equation*}
		
		By the standard contraction mapping argument on each $I_j$, starting from $Z(t_1)=0$ and iterating over the finitely many subintervals, we may choose $t_2>0$ sufficiently small, depending only on $\|\U(t_1)\|_{H_x^1(\mathbb R^3)}$, $t_0$, and $T_0$, such that
		\begin{equation*}
			Z\in C \bigl([t_1,t_1+t_2];H_x^1(\mathbb R^3)\bigr).
		\end{equation*}
		Consequently,
		\begin{equation*}
			\U\in C\bigl([t_1,t_1+t_2];H_x^1(\mathbb R^3)\bigr).
		\end{equation*}
		This completes the proof.
	\end{proof}

	Therefore, combining Propositions \ref{prop:h1gwp-lwpI}--\ref{prop:p4-endpoint}, we obtain that the equation \eqref{eq:nls-initial} admits a solution $\U\in C([0,T_0];H_x^1(\R^3))$ for any $T_0>0$. Thus, Proposition \ref{thm:gwp} follows by applying the inverse pseudo-conformal transform.
	
	\vspace{2cm}
	
	\section{Final data problem for 3D quadratic NLS}\label{sec: 3DquaNLS}
	
	\vspace{0.5cm}
	
	Applying the pseudo-conformal transform, we reduce Theorem \ref{thm:final} to the initial data problem for the following nonautonomous equation:
	\EQn{\label{eq:nls-pc-final}
		i\pd_t \U + \frac12 \De \U = t^{-\frac{1}{2}}|\U|\U,
	}
	with initial data $\U(0)\in\dot H_x^{\frac{1}{2}}(\R^3)$. It suffices to prove its well-posedness on $[0,T_0]$ for any $T_0>1$.
	
	To start with, we introduce the following setup:
	\begin{assu}\label{assu:main-final}
		We make the following assumptions.
		\begin{enumerate}
			\item
			Let $T_0>1$, $|x|^{\frac{1}{2}}u_+\in L_x^2(\R^3)$, and $u_+$ is radial. Let $\de_1>0$ be a sufficiently small absolute constant. Let $\phi\in C_0^{\infty}(\R)$ such that $\phi(r)=1$ if $|r|\leq 1$ and $\phi(r)=0$ if $|r|>2$.
			\item
			Fix a sufficiently small constant $0<\de_0:=\de_0(u_+,T_0, \delta_1)\ll 1$ that will be defined later, as specified precisely in \eqref{defn:delta0-final} below. Since $|x|^{\frac{1}{2}}u_+\in L_x^2(\R^3)$, we can choose dyadic numbers $0<R_-<1<R_+$ such that, with
			\begin{equation*}
				\rho_{R_-,R_+}(x):=(1-\phi(R_-^{-1}|x|))\phi(R_+^{-1}|x|),
			\end{equation*}
			we have
			\begin{equation}\label{eq:initialdata-estimate-v-final}
				\normb{|x|^{\frac{1}{2}}(1-\rho_{R_-,R_+})u_+}_{L_x^2(\R^3)} \le \de_0.
			\end{equation}
			\item
			Set
			\begin{equation*}
				v_+:=(1-\rho_{R_-,R_+})u_+ \text{, and }w_+:=\rho_{R_-,R_+}u_+.
			\end{equation*}
			Define
			\begin{equation*}
				v:=S(t)v_+\text{, and }w:=u-v.
			\end{equation*}
			\item 
			Apply the pseudo-conformal transform, and define
			\begin{equation*}
				\U:=\TT u\text{, }\V:=\TT v\text{, and }\W:=\TT w.
			\end{equation*}
			Denote the transformed initial data by
			\begin{equation*}
				\U_0:= \F^{-1} \wb u_+\text{, } \V_0 := \F^{-1} \wb v_+\text{, and }\W_0 := \F^{-1} \wb w_+.
			\end{equation*}
			Since $w_+$ is supported in the annulus $R_-\leq |x|\leq 2R_+$, we have $\W_0\in L_x^2(\R^3)$. We denote
			\begin{equation*}
				A_0:=1+\norm{\W_0}_{L_x^2(\R^3)}^2.
			\end{equation*}
		\end{enumerate}
		\begin{remark}
			Throughout this section, the norm $\normo{|x|^{1/2}u_+}_{L_x^2(\R^3)}$ can be viewed as a constant. We therefore suppress this dependence, and write
			\begin{equation*}
				C=C\left( \normb{|x|^{\frac12}u_+}_{L_x^2(\R^3)} \right)
			\end{equation*}
			for short.
		\end{remark}
	\end{assu}
	
	Therefore, under Assumption \ref{assu:main-final}, $w$ solves the equation
	\EQ{
		\left\{ \aligned
		&i\pd_t w +\frac12\De w=|v+w|(v+w), \\
		&\lim_{t\ra+\I}\norm{S(-t)w(t,x) - w_+}_{\dot\Sigma^{\frac12}(\R^3)} =0.
		\endaligned
		\right.
	} 
	We also have that
	\begin{equation*}
		\V = S(t)\V_0,
	\end{equation*}
	and
	\begin{equation*}
		\W = S(t)\W_0 -i\int_{0}^tS(t-\ta)(\ta^{-\frac12}|\V+\W|(\V+\W)) \dd \ta.
	\end{equation*}
	By \eqref{eq:initialdata-estimate-v-final}, we obtain the following.
	\begin{lem}[Estimate for the initial data]
		Let Assumption \ref{assu:main-final} hold. Then,
		\begin{equation}\label{eq:initialdata-estimate-v-final-2}
			\norm{\V_0}_{\dot H_x^{\frac12}(\R^3)} \le \de_0, 
		\end{equation}
		and
		\begin{equation}\label{eq:initialdata-estimate-w}
			\norm{\W_0}_{L_x^2(\R^3)} \le A_0^{\frac12}.
		\end{equation}
	\end{lem}
	Combining \eqref{eq:initialdata-estimate-v-final-2}, the Strichartz estimate \eqref{eq:strichartz-1}, and the radial Strichartz estimate \eqref{eq: radial estimate} yields the following linear estimates.
	\begin{lem}
		Let Assumption \ref{assu:main-final} hold. Then, for any admissible pair $(q,r)$,
		\begin{equation}\label{eq:linear-estimate-v-final}
			\normb{|\nabla|^{\frac12}\V}_{L_t^q L_x^r(\R\times\R^3)} + \normb{|\nabla|^{\frac12}\V}_{L_t^{q,2} L_x^r(\R\times\R^3)}\lsm \de_0.
		\end{equation}
		Furthermore, since $u_+$ is radial, we also have the endpoint Strichartz estimate
		\begin{equation}\label{eq:linear-estimate-v-radial-final}
			\norm{\V}_{L_t^2 L_x^\I(\R\times\R^3)}\lsm \de_0.
		\end{equation}
	\end{lem}
	Now, consider the Cauchy problem \eqref{eq:nls-pc-final} with initial data $\U(0)=\U_0$.
	\begin{prop}[Local theory I: local well-posedness in the critical space]\label{prop:localI-final}
		Let Assumption \ref{assu:main-final} hold. Then, there exists $t_0=t_0(\de_1,\U_0)\in (0,1)$ such that the equation \eqref{eq:nls-pc-final} with initial data $\U(0)=\U_0$ admits a unique solution 
		\begin{equation*}
			\U\in C([0,t_0]; \dot H_x^{\frac12}(\R^3)),
		\end{equation*} 
		such that
		\begin{equation*}
			\normb{|\nabla|^{\frac12}\U}_{L_t^\I L_x^2 ([0,t_0]\times\R^3)} \lsm 1,
		\end{equation*}
		and
		\begin{equation}\label{eq:local-h1/2-smallness-final}
			\normb{|\nabla|^{\frac12}\U}_{L_t^{8,2} L_x^{\frac{12}{5}}([0,t_0]\times\R^3)}   \lsm \de_1.
		\end{equation}
	\end{prop}
	\begin{proof}
		The proof follows from the standard contraction mapping argument, and we only give the key nonlinear estimate. Let $t_0>0$ be some suitable small parameter, and $I:=[0,t_0]$. By the Strichartz estimate, H\"older's inequality in Lorentz space, and Sobolev's inequality,
		\begin{equation}\label{sec5-1}
			\begin{aligned}
				&\normB{|\nabla|^{\frac 12} \int_0^t  S(t-\ta)(\ta^{-\frac12}|\U(\ta)|\U(\ta))\dd \ta}_{L_t^\I L_x^{2} \cap L_t^{8,2} L_x^{\frac{12}{5}}(I\times\R^3)} \\
				&\lsm \|t^{-\frac 12}\|_{L_t^{2,\infty}(I)}\|\U\|_{L_t^{8,2}L_x^4(I\times\R^3)} \||\nabla|^{\frac12}\U\|_{L_t^{8,2}L_x^{\frac{12}{5}}(I\times\R^3)}\\
				&\lsm  \normb{|\nabla|^{\frac12}\U}_{L_t^{8,2} L_x^{\frac{12}{5}}(I\times\R^3)}^2.
			\end{aligned}
		\end{equation}
	\end{proof}
	\begin{prop}[Local theory II: local mass estimate]\label{prop:localII-final}
		Let Assumption \ref{assu:main-final} hold and $t_0$ be given in Proposition \ref{prop:localI-final}. Then, we have $\W\in C([0,t_0]; L_x^2(\R^3))$ and
		\begin{equation}\label{eq: W-L2-t0}
			\norm{\W}_{L_t^\I L_x^2([0,t_0]\times\R^3)} \lesssim A_0^{\frac12}.
		\end{equation}
	\end{prop}
	The proof is reduced to the following nonlinear estimate. We omit the standard argument showing that Lemma \ref{lem:localII-final} implies Proposition \ref{prop:localII-final}.
	\begin{lem}\label{lem:localII-final}
		Let Assumption \ref{assu:main-final} hold, and $t_0$ be given in Proposition \ref{prop:localI-final}. Then, for any $0<t_0'\le t_0$,
		\begin{equation*}
			\left\| \int_{0}^t S(t-\ta)(\ta^{-\frac{1}{2}}|\U|\U) \dd \ta \right\|_{L_t^\I L_x^2([0,t_0']\times\R^3)} \lsm  \de_1  \norm{\W}_{L_t^\I L_x^2([0,t_0']\times\R^3)}  + \de_0\de_1.
		\end{equation*}
	\end{lem}
	\begin{proof}
		In the proof of this lemma, we restrict $(t,x)$ to $[0,t_0']\times\R^3$. By the Strichartz estimate, Sobolev's inequality, and Proposition \ref{prop:localI-final},
		\EQ{
			&\left\| \int_{0}^t  S(t-\ta)(\ta^{-\frac{1}{2}}|\U|\U) \dd \ta \right\|_{L_t^\I L_x^2} \lsm  \normb{t^{-\frac12}\U\W}_{L_{t}^{\frac85,2} L_x^{\frac43}} + \normb{t^{-\frac12}\U\V}_{L_{t}^{\frac{8}{7}} L_x^{\frac{12}{7}}} \\
			& \qquad \lsm  \|t^{-\frac12}\|_{L_t^{2,\infty}} \norm{\U}_{L_t^{8,2} L_x^4} \norm{\W}_{L_t^\I L_x^2} + \|t^{-\frac12}\|_{L_t^{\frac43}} \norm{\U}_{L_t^{8} L_x^4} \norm{\V}_{L_t^\I L_x^3} \\
			& \qquad \lsm  \de_1 \norm{\W}_{L_t^\I L_x^2} + \de_1\de_0.
		}
		This finishes the proof.
	\end{proof}
	
	Similarly, for $t_1>0$, we consider the Cauchy problem for the equation
	\begin{equation}\label{eq:nls-w-localIII-final}
		\W(t)=S(t-t_1)\W(t_1)-i\int_{t_1}^t S(t-\tau)\big( \ta^{-\frac12}|\U(\ta)|\U(\ta)\big) \dd \ta,
	\end{equation}
	which is the equation used to extend the $L_x^2$ component.
	\begin{prop}[Local theory III: extension of the solution]\label{prop:localIII-final}
		Let Assumption \ref{assu:main-final} hold and $t_0$ be given in Proposition \ref{prop:localI-final}. Then, for any $t_1\in [t_0,T_0)$, there exists $0<t_2\le\min\fbrk{1,T_0-t_1}$ depending on $\norm{\W(t_1)}_{L_x^2(\R^3)}$ and $t_0$, such that the equation \eqref{eq:nls-w-localIII-final} admits a unique solution $\W \in C([t_1,t_1+t_2];L_x^2(\R^3))$. Moreover, for any admissible pair \((q,r)\), 
		\begin{equation*}
			\|\W\|_{L_t^qL_x^r([t_1,t_1+t_2]\times \R^3)} \leq C(\norm{\W(t_1)}_{L_x^2(\R^3)}, t_0).
		\end{equation*}
	\end{prop}
	\begin{proof}
		Take some $t_2$ that will be defined later such that $0<t_2\le\min\fbrk{1,T_0-t_1}$. In the proof of this proposition, we denote $I:=[t_1,t_1+t_2]$, and restrict $(t,x)$ to $I\times\R^3$. 
		
		By the Strichartz estimate and H\"older's inequality in $t$, for any admissible pair $(q,r)$,
		\EQ{
			\normB{\int_{t_1}^t S(t-\tau) \ta^{-\frac12}|\U(\ta)|\U(\ta)\dd \ta}_{L_t^qL_x^r} & \lsm  t_0^{-\frac12} \norm{\V+\W}^2_{L_t^2 L_x^4} \\
			&\lsm t_0^{-\frac12} t_2^{\frac 34}\|\V\|^2_{L_t^8L_x^4} + t_0^{-\frac12} t_2^{\frac14}\|\W\|^2_{L_t^{\frac83}L_x^4} \\ 
			&\lsm t_0^{-\frac12} t_2^{\frac 34}\delta_0^2 + t_0^{-\frac12} t_2^{\frac14}\|\W\|^2_{L_t^{\frac83}L_x^4}.
		}
		Then, this proposition follows by the standard contraction mapping argument.
	\end{proof}
	
	Define
	\begin{equation*}
		\M(t) := \norm{\W(t)}_{L_x^2(\R^3)}^2.
	\end{equation*}
	\begin{prop}[A priori estimate]\label{prop:mass}
		Let Assumption \ref{assu:main-final} hold. Assume that there exists $K\geq 1$ such that
		\begin{equation}\label{eq:mass-assumption}
			\M(t_0) \le KA_0.
		\end{equation}
		Take some $t_0< t_3\leq T_0$ such that $\W\in C([t_0,t_3];L_x^2(\R^3))$. Then, we have
		\begin{equation*}
			\sup_{t\in[t_0,t_3]}\M(t)\loe 2KA_0.
		\end{equation*}
	\end{prop}
	\begin{proof}
		Let $I=[t_0,t_3]$. In the proof of this proposition, all the space-time norms are taken over $I\times\R^3$. We implement a bootstrap procedure on $I$: assume an a priori bound 
		\begin{equation}\label{eq:bound-w-hypothesis-final}
			\sup_{t\in I} \M(t)\loe 2KA_0,
		\end{equation}
		then it suffices to prove that
		\begin{equation}\label{eq:bound-w-conclusion-final}
			\sup_{t\in I} \M(t)\loe \frac32 KA_0.
		\end{equation}
		Recall the equation for $\W$:
		\begin{equation*}
			i\pd_t \W + \frac12\De \W = t^{-\frac{1}{2}} |\U|\U.
		\end{equation*}
		Multiplying the equation by $\wb{\W}$, integrating in $x$, and taking the imaginary part, we obtain 
		\begin{equation*}
			\frac12\frac{\dd}{\dd t}\int_{\R^3}  |\W|^2 \dd x = t^{-\frac {1}2} \im\int_{\R^3} (|\U|\U - |\W|\W) \wb\W \dd x.
		\end{equation*}
		Integrating from \(t_0\) to \(t\in I\), we obtain
		\EQ{
			\M(t) \le & \M(t_0) + C\absb{ \int_{t_0}^{t}\int_{\R^3} \ta^{-\frac {1}2}(|\U|\U - |\W|\W) \wb\W \dd x\dd\ta} \\
			\le & KA_0 + C\absb{ \int_{t_0}^{t}\int_{\R^3} \ta^{-\frac {1}2}(|\V|+ |\W|) |\W\V| \dd x\dd\ta}.
		}
		
		By \eqref{eq:linear-estimate-v-radial-final} and \eqref{eq:bound-w-hypothesis-final}, for $\delta_0\leq 1$, we have that
		\begin{equation*}
			\absb{ \int_{t_0}^{t}\int_{\R^3} \ta^{-\frac {1}2} |\W|^2|\V| \dd x\dd\ta} \lsm_{t_0, T_0} \norm{\V}_{L_t^2 L_x^\I} \norm{\W}_{L_t^\I L_x^2}^2 \lsm_{t_0, T_0} KA_0\de_0,
		\end{equation*}
		and
		\EQ{
			\absb{ \int_{t_0}^{t}\int_{\R^3} \ta^{-\frac {1}2} |\W||\V|^2 \dd x\dd\ta} &\lsm_{t_0, T_0} \norm{\V}_{L_t^8 L_x^4}^2 \norm{\W}_{L_t^\I L_x^2} \\
			&\lsm_{t_0,T_0} K^{\frac12}A_0^{\frac12}\de_0^2\lsm_{t_0, T_0} KA_0\de_0.
		}
		Therefore
		\begin{equation*}
			\sup_{t\in I}\M(t)\loe KA_0 + C(t_0,T_0)\delta_0 KA_0.
		\end{equation*}
		Then, \eqref{eq:bound-w-conclusion-final} follows by taking $\de_0=\de_0(T_0,t_0)$ such that
		\begin{equation}\label{defn:delta0-final}
			\de_0 \ll \min\fbrkbb{1,\frac{1}{2C(t_0,T_0)}}.
		\end{equation}
		It is easy to see that $\de_0$ depends on $T_0$, $u_+$ and $\delta_1$.
	\end{proof}
	
	Thus, by the standard extension argument, we obtain 
	\begin{equation}\label{sec5-2}
		\W\in C([t_0,T_0]; L_x^2(\R^3)) \cap L_t^qL_x^r([t_0, T_0]\times \R^3),
	\end{equation}
	for any admissible pair \((q,r)\). Since \(\U-S(t)\U_0=\W-S(t)\W_0 \in C([t_0,T_0]; L_x^2(\R^3)) \), by Proposition \ref{prop:localI-final}, it remains to show 
	\begin{equation}\label{sec5-3}
		\U-S(t)\U_0 \in C([t_0,T_0]; \dot{H}_x^{\frac{1}{2}}(\R^3)).
	\end{equation}
	Let $I=[\tilde{t}, \tilde{t} +\sigma]$, where $\tilde{t}\in [t_0, T_0)$. We restrict $(t, x)$ to $I\times\R^3$. Arguing as in \eqref{sec5-1},
	\begin{align*}
		&\||\nabla|^{\frac{1}{2}}\U(t)\|_{L_t^\I L_x^{2} \cap L_t^{8,2} L_x^{\frac{12}{5}}} \\
		&\lesssim  \||\nabla|^{\frac{1}{2}}\U(\tilde{t})\|_{L_x^{2}} +  \|t^{-\frac{1}{2}}(\V+\W)\|_{L_t^{\frac{8}{5}, 2}L_x^4} \||\nabla|^{\frac{1}{2}}\U\|_{L_t^{8,2}L_x^{\frac{12}{5}}}\\
		& \lesssim \||\nabla|^{\frac{1}{2}}\U(\tilde{t})\|_{L_x^{2}} + (\delta_0 + \|t^{-\frac{1}{2}}\|_{L_t^4(I)} \|\W\|_{L_t^{\frac{8}{3}}L_x^4})  \||\nabla|^{\frac{1}{2}}\U\|_{L_t^{8,2}L_x^{\frac{12}{5}}}.
	\end{align*} 
	By \eqref{sec5-2}, we can choose $\sigma$ sufficiently small independent of $\wt{t}$ such that 
	\begin{equation*}
		\||\nabla|^{\frac{1}{2}}\U\|_{L_t^\I L_x^{2} \cap L_t^{8,2} L_x^{\frac{12}{5}}} \leq C \||\nabla|^{\frac{1}{2}}\U(\tilde{t})\|_{L_x^{2}} + \frac{1}{2}  \||\nabla|^{\frac{1}{2}}\U\|_{L_t^\I L_x^{2} \cap L_t^{8,2}L_x^{\frac{12}{5}}}.
	\end{equation*}
	Thus, by the standard iterative procedure, we have \(\U \in C([t_0,T_0]; \dot{H}_x^{\frac{1}{2}}(\R^3))\). Hence, \eqref{sec5-3} holds. Applying the inverse pseudo-conformal transform, we obtain Theorem \ref{thm:final}.

\end{document}